\documentclass[12pt]{article}
\usepackage[margin=1in]{geometry}
\usepackage[utf8]{inputenc}

\usepackage[english]{babel}
\usepackage[round]{natbib}
\usepackage[dvipsnames]{xcolor}

\usepackage{amsmath, amssymb, amsthm,mathtools,dsfont}
\usepackage{bbm}
\usepackage{natbib}
\usepackage{blkarray}
\usepackage{cancel}
\usepackage{enumitem}
\usepackage{float}
\usepackage{bm}
\usepackage{graphicx}
\usepackage{mathrsfs}
\usepackage{microtype}
\usepackage{ragged2e}
\usepackage{sectsty}
\usepackage{thmtools}
\usepackage{tikz}
\usepackage[Symbolsmallscale]{upgreek}

\usepackage{microtype}
\usepackage{graphicx}
\usepackage{subfigure}
\usepackage{hyperref}

\hypersetup{citecolor=purple, colorlinks=true, linkcolor=blue}

\newcommand{\cC}{\mathcal{C}}

\newcommand{\cF}{\mathcal{F}}

\newcommand{\cH}{\mathcal{H}}
\newcommand{\cI}{\mathcal{I}}

\newcommand{\cK}{\mathcal{K}}

\newcommand{\cM}{\mathcal{M}}
\newcommand{\cN}{\mathcal{N}}

\newcommand{\cP}{\mathcal{P}}

\newcommand{\cV}{\mathcal{V}}

\newcommand{\NN}{\mathbb{N}}

\newcommand{\WW}{\mathbb{W}}

\newcommand{\kl}[2]{\KL(#1 \!\;\|\; \!#2)}

\newcommand*{\norm}[1]{\left\|#1\right\|}
\newcommand*{\triplenorm}[1]{{\left\vert\kern-0.25ex\left\vert\kern-0.25ex\left\vert #1
    \right\vert\kern-0.25ex\right\vert\kern-0.25ex\right\vert}}

\DeclareMathOperator{\id}{id}

\newcommand{\R}{\mathbb{R}}
\newcommand{\Rd}{\mathbb{R}^d}
\renewcommand{\phi}{\varphi}
\newcommand{\eps}{\varepsilon}

\newcommand*{\E}{\mathbb E}

\DeclareMathOperator{\tr}{tr}

\newcommand*{\defeq}{\coloneqq}
\newcommand*{\rd}{\mathrm{d}}
\newcommand*{\dd}{\, \rd}
\DeclareMathOperator*{\argmin}{arg\,min}

\DeclareMathOperator*{\KL}{KL}

\newcommand{\inner}[1]{\left\langle#1\right\rangle}

\newcommand{\man}{\EuScript{M}}

\DeclareMathOperator{\cone}{cone}
\DeclareMathOperator{\conv}{conv}
\DeclareMathOperator{\Cov}{Cov}
\DeclareMathOperator{\diag}{diag}
\DeclareMathOperator{\diam}{diam}
\DeclareMathOperator{\proj}{proj}
\newcommand{\Pdiam}{\cP_\diamond}

\usepackage{algorithm}
\usepackage{algpseudocode}
\usepackage[mathscr]{euscript}
\usepackage{fancyhdr}

\usepackage[capitalize,noabbrev]{cleveref}

\theoremstyle{plain}
\newtheorem{theorem}{Theorem}[section]
\newtheorem{proposition}[theorem]{Proposition}

\newtheorem{lemma}[theorem]{Lemma}
\newtheorem{corollary}[theorem]{Corollary}
\theoremstyle{definition}
\newtheorem{definition}[theorem]{Definition}

\theoremstyle{remark}
\newtheorem{remark}[theorem]{Remark}

\newcommand{\coneM}{\cone(\mathcal{M})}
\newcommand{\convM}{\conv(\mathcal{M})}

\newcommand{\augconeM}{\underline{\cone}(\mathcal{M})}
\newcommand{\augtip}{\underline{\cone}(\mathcal{M};\, \alpha\id)}
\newcommand{\dLOT}{\mathsf{d}_{\mathrm{LOT},\rho}}

\title{Algorithms for mean-field variational inference \\
via polyhedral optimization in the Wasserstein space}
\author{
Yiheng Jiang\thanks{Courant Institute of Mathematical Sciences, New York University. \tt yj2070@nyu.edu} \and 
Sinho Chewi\thanks{Department of Statistics and Data Science, Yale University. \tt{sinho.chewi@yale.edu}} \and 
Aram-Alexandre Pooladian\thanks{Center for Data Science, New York University. \tt aram-alexandre.pooladian@nyu.edu} 
}

\begin{document}
\maketitle

\begin{abstract}
We develop a theory of finite-dimensional polyhedral subsets over the Wasserstein space and optimization of functionals over them via first-order methods. 
Our main application is to the problem of mean-field variational inference (MFVI), which seeks to approximate a distribution $\pi$ over $\R^d$ by a product measure $\pi^\star$.
When $\pi$ is strongly log-concave and log-smooth, we provide (1) approximation rates certifying that $\pi^\star$ is close to the minimizer $\pi^\star_\diamond$ of the KL divergence over a \emph{polyhedral} set $\Pdiam$, and (2) an algorithm for minimizing $\kl{\cdot}{\pi}$ over $\Pdiam$ {based on accelerated gradient descent over $\R^d$. As a byproduct of our analysis, we obtain the first end-to-end analysis for gradient-based algorithms for MFVI.}
\end{abstract}

\section{Introduction}\label{sec: intro}

This paper develops a framework for optimizing over \emph{polyhedral} subsets of the Wasserstein space, with accompanying guarantees. Our main application is to provide the first end-to-end computational guarantees for mean-field variational inference \citep{wainwright2008graphical, blei2017variational} under standard tractability assumptions on the posterior distribution.
We now contextualize our work with respect to the broader literature.

Optimization over (subsets of) the Wasserstein space (the metric space of probability measures over $\R^d$ endowed with the $2$-Wasserstein distance, see \cref{sec: background_ot}) has found diverse and effective applications in modern machine learning. Notable examples include distributionally robust optimization \citep{kuhn2019wasserstein, yue2022linear}, the computation of barycenters \citep{cuturi2014fast, ZemPan19Frechet, Chewietal20BWBary, altschuler2021averaging, Bacetal22Bary}, sampling~\citep{JKO, Wib18SamplingOpt, ChewiBook}, and variational inference (see below).
The development of optimization algorithms over this space, however, has been hindered by significant implementation challenges stemming from its infinite-dimensional nature and the curse of dimensionality which impedes efficient representation of high-dimensional distributions.

To alleviate these hurdles, a popular approach is to restrict the optimization to tractable subfamilies of probability distributions, such as finite-dimensional parametric families. Note that this is in contrast to Euclidean optimization, in which constraint sets are typically imposed as part of the problem
(e.g., affine constraints in operations research). Here we view the use of a constraint set in the Wasserstein space as a \emph{design choice}, with the end goals of flexibility, interpretability, and computational tractability. 

An important motivating example is that of variational inference (VI), which seeks the best approximation to a probability measure $\pi$ over $\R^d$ in the sense of KL divergence over some subset of probability measures $\cC$:
\begin{align}\label{eq:vi}
    \pi^\star \in \argmin_{\mu \in \cC} \kl{\mu}{\pi} = \argmin_{\mu \in \cC} \int \log \Bigl(\frac{\!\dd \mu}{\! \dd \pi}\Bigr) \dd \mu\,.
\end{align}
For example, $\cC$ could be taken to be the class of non-degenerate Gaussian distributions, in which case~\eqref{eq:vi} is known as Gaussian VI.
Recently, by leveraging the rich theory of gradient flows over the Wasserstein space,~\cite{lambert2022variational,diao2023forward} provided algorithmic guarantees for Gaussian VI under standard tractability assumptions, i.e., strong log-concavity and log-smoothness of $\pi$.

In this paper, we instead study the problem of \emph{mean-field} VI\@, in which $\cC$ is taken to be the class of product measures over $\R^d$, written $\cP(\R)^{\otimes d}$.
In this context, the works~\cite{zhang2020theoretical,yao2022mean, lacker2023independent} 
have also developed algorithms based on Wasserstein gradient flows, although computational guarantees for VI are still nascent (see~\Cref{sec: mfvi} for further details and comparison with the literature).

The main result of our work is to provide computational guarantees under the usual tractability assumptions for $\pi$.
Our approach is to replace the set of product measures by a smaller, ``\emph{polyhedral}'' subset $\Pdiam$ which we prove is an accurate approximation to $\cP(\R)^{\otimes d}$, in the sense that the minimizer $\pi^\star$ of~\eqref{eq:vi} is in fact close to the KL minimizer $\pi^\star_\diamond$ over $\Pdiam$ with quantifiable approximation rates.
This motivates our development of a theory of polyhedral optimization over the Wasserstein space which, when applied to the mean-field VI problem, furnishes algorithms for minimization of the KL divergence over $\Pdiam$ with theoretical (even accelerated) guarantees.
More broadly, we are hopeful that the success of polyhedral optimization for mean-field VI will encourage the further use of polyhedral constraint sets to model other problems of interest. 

We discuss the implementation of our algorithm in~\Cref{sec: implementation_details}, with code available \href{https://github.com/APooladian/MFVI}{here}.
Below, we describe our contributions in more detail.

\subsection*{Main contributions}

\paragraph*{Polyhedral optimization in the Wasserstein space.}
We study parametric sets of the following form:
\begin{align*}
    {\coneM}_\sharp \rho \defeq \Bigl\{ \bigl(\textstyle \sum_{T \in \cM} \lambda_T T\bigr)_\sharp \rho \Bigm\vert \lambda \in \R^{|\cM|}_+ \Bigr\}\,,
\end{align*}
where $\R^{|\cM|}_+$ is the non-negative orthant, $\cM$ is a family of user-chosen optimal transport maps, and $\rho$ is a fixed, known, reference measure.
To our knowledge, such sets have not previously appeared in the literature.

Before proceeding, however, we must dispel a potential source of confusion: although $\coneM$ is a convex subset of the space of optimal transport maps at $\rho$---in other words, a convex subset of the \emph{tangent space} to Wasserstein space at $\rho$---the set $\coneM_\sharp \rho$ is \emph{not always} a convex subset of the Wasserstein space itself, in the sense of  being closed under Wasserstein geodesics.
For this to hold, we impose a further 
 condition on $\cM$, known as \emph{compatibility} \citep{boissard2015distribution}.
 Although compatibility is restrictive, it is nevertheless powerful enough to capture our application to mean-field VI described below.
 We refer to the set $\coneM_\sharp \rho$, for a compatible family $\cM$, as a \emph{polyhedral} subset of the Wasserstein space (or more specifically, a cone).

The assumption of compatibility entails strong consequences: we show that in fact, $(\coneM_\sharp \rho, W_2)$ is isometric
to $(\R^{|\cM|}_+, \|\cdot\|_Q)$, where $\|\cdot\|_Q$ is a \emph{Euclidean} norm.
This isometry allows us to optimize functionals over $\coneM_\sharp \rho$ via lightweight first-order algorithms for Euclidean optimization in lieu of Wasserstein optimization, which often requires computationally burdensome approximation schemes such as interacting particle systems.
In particular, we can apply projected gradient descent or incorporate faster, \emph{accelerated} methods.
Moreover, under the isometry, convex subsets of $\coneM$ map to convex subsets of $\coneM_\sharp \rho$, giving rise to a bevy of geodesically convex constraint sets over which tractable optimization is feasible.
This includes Wasserstein analogues of polytopes, to which we can apply the projection-free Frank--Wolfe algorithm.
We show that as soon as the objective functional $\cF$ is geodesically convex and smooth, these algorithms inherit the usual rates of convergence from the convex optimization literature.

\paragraph*{Application to mean-field VI\@.}
We next turn toward mean-field VI as a compelling application of our theory of polyhedral optimization. Throughout, we only assume that $\pi$ satisfies the standard assumptions of strong log-concavity and log-smoothness.
By leveraging the structure of the mean-field VI solution and establishing regularity bounds for optimal transport maps between well-conditioned product measures, we first prove an approximation result which shows that the solution $\pi^\star$ to mean-field VI in~\eqref{eq:vi} is well-approximated by the minimizer $\pi_\diamond^\star$ of the KL divergence over a suitable polyhedral approximation $\Pdiam$ of the space of product measures.
Importantly, our approximation rates, owing to the coordinate-wise decomposability of mean-field VI\@, do not incur the curse of dimensionality. 

Next, we establish the geodesic strong convexity and geodesic smoothness of the KL divergence over $\Pdiam$.
Consequently, bringing to bear the full force of the Euclidean--Wasserstein equivalence, we obtain, to the best of our knowledge, the first \emph{end-to-end} convergence rates for mean-field VI\@.

\paragraph*{Organization.}
Our paper is outlined as follows: In \Cref{sec: background_ot} we provide relevant background information on optimal transport theory and the Wasserstein space.
In \Cref{sec: wass_polytopes}, we develop known properties of compatibility which provide tools for building large compatible families, and we establish the key isometry with $(\R^{|\cM|}_+, \|\cdot\|_Q)$. 
We prove our general continuous- and discrete-time guarantees for Wasserstein polyhedral optimization in \Cref{sec: algorithms}.
We apply our framework to mean-field variational inference in \Cref{sec: mfvi_main}; in particular, the implementation is discussed in~\Cref{sec: implementation_details}. \Cref{sec:experiments} contains our numerical experiments. \Cref{sec:mixtures_of_prod} extends our framework to mixtures of product measures. We conclude with numerous open directions. \looseness-1

\subsubsection*{Related work}

To the best of our knowledge, our introduction of polyhedral sets and theory of polyhedral optimization over the Wasserstein space are novel. A special case of our set is 
\begin{align*}
    \convM_\sharp \rho \defeq \Bigl\{ \bigl(\textstyle \sum_{T \in \cM} \lambda_T T \bigr)_\sharp \rho \Bigm\vert \lambda \in \Delta_{|\cM|} \Bigr\}\,,
\end{align*}
where $\Delta_{|\cM|}$ is the $|\cM|$-simplex. Such a constraint set is used by \cite{boissard2015distribution, gunsilius2022tangential, werenski2022measure}, and is usually studied in the context of Wasserstein barycenters. The work of \cite{bonneel2016wasserstein} also considers $\convM_\sharp \rho$, but makes no assumptions on the maps, and they tackle the problem from a computational angle via Sinkhorn's algorithm \citep{cuturi2013sinkhorn,PeyCut19}, albeit without convergence guarantees. \cite{Alb+24StochInterpolants} use the same set, but without incorporating any optimal transport theory.

Our approach to mean-field VI\@, which parameterizes the variational family as the pushforward of a reference measure via transport maps, has its roots in the literature on generative modeling and normalizing flows \citep{chen2018neural, finlay2020learning, finlay2020train, Huaetal21CvxFlows}.
We provide further background information and literature on mean-field VI in \cref{sec: mfvi}, and omit it here to avoid redundancies.

{Finally, we mention that our work falls under the category of \emph{linearized optimal transport} \citep{wang2013linear}, which we closely address in \cref{sec: isometry}.}

\section{Background on optimal transport}\label{sec: background_ot}

In this section, we provide background on optimal transport relevant to our work and refer to~\citet{Vil08, San15} for details.
Throughout, we assume that all probability measures admit a density function with respect to Lebesgue measure.
We let $\cP_{2}(\R^d)$ denote the set of probability measures with density over $\R^d$ with finite second moment.

For $\rho,\mu \in \cP_{2}(\R^d)$, the squared $2$-Wasserstein distance is written as 
\begin{align}\label{eq: wass_def}
    W_2^2(\rho,\mu) = \inf_{T : T_\sharp\rho=\mu} \int \|x -T(x)\|^2_2 \dd \rho(x)\,,
\end{align}
where the collection $\{T\!: \!T_\sharp\rho=\mu\}$ is the set of all valid transport maps: for $X \sim \rho$, $T(X)\sim\mu$. 

Since we assumed $\rho$ has a density, Brenier's theorem \citep{Bre91} states that there exists a unique minimizer to \cref{eq: wass_def}, called the \emph{optimal transport map} $T_\star$ between $\rho$ and $\mu$. Further, $T_\star = \nabla \phi_\star$ for some convex function $\phi_\star$, called a Brenier potential. 

Additionally, since $\mu$ also has a density, then there exists an optimal transport map between $\mu$ and $\rho$, given by $\nabla \phi_\star^* = (T_\star)^{-1}$, where $\phi_\star^*(y) \defeq \sup_{x\in\Rd}\{\inner{x,y} - \phi_\star(x)\}$ is the Fenchel conjugate of $\phi_\star$. 
For more information on (differentiable) convex functions and conjugacy, we suggest \cite{rockafellar1997convexanalysis,hiriart2004fundamentals}.

Recall that a function $f : \R^d \to \R$ is \emph{$m$-strongly convex} in some norm $\|\cdot\|$ if 
\begin{align*}
    f(y) \geq f(x) + \langle \nabla f(x), y - x \rangle + \frac{m}{2}\,\|x - y\|^2\,, \qquad x,y\in\R^d\,,
\end{align*}
and \emph{$M$-smooth} in some norm $\|\cdot\|$  if
\begin{align*}
    f(y) \leq f(x) + \langle \nabla f(x), y - x \rangle + \frac{M}{2}\,\|x - y\|^2\,, \qquad x,y\in\R^d\,,
\end{align*}
where $m, M > 0$.

For two probability measures $\mu_0, \mu_1 \in \cP_2(\R^d)$, let $\nabla \phi^{0\to1}$ denote the optimal transport map from $\mu_0$ to $\mu_1$. The \emph{(unique) constant-speed geodesic} between $\mu_0$ and $\mu_1$ is given by the curve $(\mu_t)_{t \in [0,1]}$, with
\begin{align}\label{eq: constant_speed_geod}
    \mu_t = (\nabla \phi_t)_\sharp \mu_0 \defeq (\text{id} + t\, (\nabla \phi^{0\to1}-\text{id}))_\sharp \mu_0\,.
\end{align}

If we equip $\cP_2(\Rd)$, the space of probability distributions with finite second moment over $\Rd$, with the $2$-Wasserstein distance, we obtain a metric space $\WW \defeq (\cP_2(\Rd),W_2)$~\citep[Theorem 7.3]{villani2021topics}, which we call the \emph{Wasserstein space}. In fact, it can be formally viewed as a Riemannian manifold over which one can define gradient flows of functionals \citep{otto2001geometry}. We refer the interested reader to consult the background sections of \citet{altschuler2021averaging} or to~\cite{ChewiBook} for a light exposition and further details.

The Riemannian structure of the Wasserstein space is crucial for the development of optimization over this space, as it furnishes appropriate Wasserstein analogues of basic concepts from Euclidean optimization, such as the gradient mapping, convexity, and smoothness.
In particular, we say that a subset $\cC$ of the Wasserstein space is \emph{geodesically convex} if it is closed under taking geodesics~\eqref{eq: constant_speed_geod}.
Also, a functional $\cF : \cP_2(\R^d)\to\R$ is \emph{geodesically (strongly) convex} (resp.\ \emph{geodesically smooth}) if the map $[0,1] \to \R$, $t\mapsto \cF(\mu_t)$ is (strongly) convex (resp.\ smooth) along every constant-speed geodesic $(\mu_t)_{t\in [0,1]}$.

\section{Polyhedral sets in the Wasserstein space}\label{sec: wass_polytopes}

In this section, we establish properties of the constraint set
\begin{align}\label{eq: cj_def}
    \coneM_\sharp\rho \defeq \Bigl\{ \bigl(\textstyle \sum_{T \in \cM}\lambda_T T \bigr)_\sharp \rho \Bigm\vert \lambda \in \R^{|\cM|}_+ \Bigr\}\,,
\end{align}
with respect to the known base measure $\rho$ and a fixed set of optimal transport maps $\cM$.
Typically, we have in mind finite $\cM$, in which case~\eqref{eq: cj_def} is valid. Otherwise,~\eqref{eq: cj_def} should be modified to range only over $\lambda$ with finitely many non-zero coordinates, or in other words, $\coneM$ is the smallest set containing all conic combinations of maps in $\cM$.

Despite its simplicity, we argue that the geometry of $\coneM_\sharp\rho$ is surprisingly deceptive. Most strikingly, it is \emph{not} always a geodesically convex set. Consider $T_1(x) = x$, $T_2(x) = A^{1/2}x$, and $T_3(x) = B^{1/2}x$, with $\rho = \cN(0,I)$, the standard Gaussian in $\R^d$, and $A,B \succ 0$. In this setting, $\coneM_\sharp \rho$ is the following set of Gaussians:
\begin{align}\label{eq: cj_gaussians}
    \coneM_\sharp \rho = \bigl\{\cN(0,\,(\lambda_1 I + \lambda_2 A^{1/2} + \lambda_3B^{1/2})^2) \bigm\vert \lambda \in \R_+^3\bigr\}\,.
\end{align}
One can check with virtually any randomly generated positive definite matrices $A$ and $B$ that, as long as all three matrices $I,A,B$ are not mutually diagonalizable, the geodesic between $\cN(0,A)$ and $\cN(0,B)$ does not lie in~\eqref{eq: cj_gaussians}. This simple example illustrates that some care is required in order to define convex constraint sets in the Wasserstein space.

\subsection{Compatible families of transport maps}

In the Gaussian example above, geodesic convexity of $\coneM_\sharp \rho$ is recovered if we additionally assume that $I$, $A$, and $B$ are mutually diagonalizable. This reflects a certain property of the maps $T_1$, $T_2$, $T_3$, which can be generalized to a property known as \emph{compatibility} \citep{boissard2015distribution}. We recall its definition and basic properties in the sequel. As always, we assume that $\rho$ admits a density with respect to Lebesgue measure.

Let $\cM$ be a set of bijective vector-valued maps, given by gradients of convex functions. We call the set of maps $\cM$ \emph{compatible} if 
\begin{align*}
    \text{for all}~T_1, T_2 \in \cM,\quad T_1 \circ (T_2)^{-1}~\text{is the gradient of a convex function.}
\end{align*}
Compatibility is a fundamental notion which lies at the heart of numerous other works~\citep[see][]{ boissard2015distribution,bigot2017geodesic, PanZem16Point, Cazetal18GeoPCA, Chewietal21Splines, werenski2022measure}. 
See \cite{panaretos2020invitation} for details. 

The main motivation for compatibility is the following theorem.

\begin{theorem}[Compatibility induces geodesic convexity]\label{thm: comp_geodconv} Suppose that $\cM$ is compatible. Then, $\coneM_\sharp\rho$ is a geodesically convex set. Moreover, for any convex subset $\cK \subseteq \coneM$, the set $\cK_\sharp \rho$ is a geodesically convex set.
\end{theorem}

Although this result is not difficult to prove, we were unable to find it in the existing literature.
In fact, it follows as a direct consequence of the isometry established in~\Cref{sec: isometry}, which will show that $\cone(\cM)_\sharp \rho$ is isometric to a convex subset of a Hilbert space.

Motivated by this theorem, we propose the following definition.

\begin{definition}\label{defn:polyhedral}
    Let $\cM$ be a compatible and finite family of optimal transport maps. We refer to $\coneM_\sharp \rho$ as a \emph{polyhedral set} in the Wasserstein space.
\end{definition}

More generally, a polyhedral set in the Wasserstein space is a set of the form $\cK_\sharp \rho$ where $\cK \subseteq \coneM$ is polyhedral  and $\cM$ is a compatible family.

The next sequence of lemmas furnish important examples of compatible families, which we prove in \cref{app: main_proofs}. 

\begin{lemma}[Mutually diagonalizable linear maps]\label{lem: mut_diag}
    Let $\cM$ be a family of mutually diagonalizable and positive definite linear maps $\R^d\to\R^d$.
    Then, $\cM$ is a compatible family.
\end{lemma}

\begin{lemma}[Radial maps]\label{lem: radial}
    Let
    \begin{align*}
        \cM \defeq \{x \mapsto g(\norm x_2)\,x \mid g : \R_+\to\R_+~\text{is continuous and strictly increasing}\}\,.
    \end{align*}
    Then, $\cM$ is a compatible family.
\end{lemma}

\begin{lemma}[One-dimensional maps]\label{lem: 1d_maps}
    Let $\cM$ denote the family of continuous and increasing\footnote{Technically, $\cM$ does not consist of \emph{bijective} maps, which we required in the definition of compatibility. In one dimension, however, the notion of compatibility still makes sense once we replace the inverse function with the quantile function.} functions $\R\to\R$.
    Then, $\cM$ is a compatible family.
\end{lemma}

\begin{lemma}[Direct sum]\label{lem: direct_sum}
    Let $\cM_1$ and $\cM_2$ be compatible families of maps on $\R^{d_1}$ and $\R^{d_2}$ respectively.
    Then, $\cM \defeq \{(x_1,x_2) \mapsto (T_1(x_1), T_2(x_2)) \mid T_1 \in \cM_1, \; T_2 \in \cM_2\}$ is a compatible family of maps on $\R^{d_1+d_2}$.
\end{lemma}

\begin{lemma}[Adding the identity]\label{lem:add_id}
    Let $\cM$ be a compatible family.
    Then, $\cM \cup \{\mathrm{id}\}$ is a compatible family.
\end{lemma}

\begin{lemma}[Adding translations]\label{lem:translations}
    Let $\cM$ be a compatible family of maps on $\R^d$.
    Then, $\{x\mapsto T(x) + v \mid T \in \cM, \; v\in\R^d\}$ is a compatible family of maps.
\end{lemma}

\begin{lemma}[Cones]\label{lem:cones}\label{lem: cones}
    Let $\cM$ be a compatible family.
    Then, $\cone(\cM)$ is a compatible family.
\end{lemma}

In the sequel, we will use these results in order to build rich compatible families, especially with an eye toward approximating coordinate-wise separable maps which arise in mean-field VI (see~\Cref{sec: mfvi_approx}).
In particular,~\Cref{lem: cones} is the starting point for the development of our theory of polyhedral optimization in the Wasserstein space.

\begin{remark}\label{rmk:generator_of_direct_sum}
    In our applications of interest, $\coneM$ is typically constructed as follows: let $\cM_1,\dotsc,\cM_d$ be univariate compatible families (\Cref{lem: 1d_maps}).
    We then take $\coneM$ to be the cone generated by the direct sum of $\cM_1,\dotsc,\cM_d$ via~\Cref{lem: direct_sum}. It is easy to see that a generating family of this cone is the set of maps $x \mapsto (0,\dotsc,0, T_i(x_i),0,\dotsc,0)$, where $T_i \in \cM_i$.
    This is a finite family of size $\sum_{i=1}^d |\cM_i|$.
\end{remark}

\subsection{Isometry with Euclidean geometry}\label{sec: isometry}

A key consequence of compatibility is that the Wasserstein distance equals the \emph{linearized optimal transport distance} with respect to $\rho$, i.e., for $T,\tilde T \in \cM$,
\begin{align}
    \dLOT^2(T_\sharp \rho, \tilde T_\sharp \rho) \defeq \|\tilde T - T\|^2_{L^2(\rho)} = \|\tilde T \circ T^{-1} - \text{id}\|^2_{L^2(T_\sharp \rho)} = W_2^2(T_\sharp \rho,\tilde T_\sharp \rho)\,, 
\end{align}
where we applied compatibility in the last equality to argue that $\tilde T \circ T^{-1}$ is the optimal transport map from $T_\sharp \rho$ to $\tilde T_\sharp \rho$.
This equality shows that for compatible $\cM$, the geometry of $\coneM_\sharp \rho$ is in a sense trivial, being isometric to a convex subset of the Hilbert space $L^2(\rho)$.
This fundamental property lies at the heart of the widespread usage of one-dimensional optimal transport in applications, see \citet{wang2013linear,basu2014detecting,kolouri2015transport,kolouri2016continuous,park2018representing,cai2020linearized,khurana2023supervised} for applications.

Next, we consider a family of the form $\cone(\cM)$, where $\cM$ is finite.
By its very definition, $\cone(\cM)$ is naturally parameterized by the non-negative orthant.
Henceforth, we write
\begin{align*}
    T^\lambda \defeq \sum_{T \in \cM} \lambda_T T\,, \qquad \mu_\lambda \defeq (T^\lambda)_\sharp \rho\,.
\end{align*}
We can therefore consider the induced metric on $\R_+^{|\cM|}$.
A straightforward calculation reveals:
\begin{align*}
    \dLOT^2(\mu_\eta, \mu_\lambda) = \|\textstyle\sum_{T\in\cM} (\eta_T - \lambda_T)\,T\|^2_{L^2(\rho)} = (\eta - \lambda)^\top Q(\eta-\lambda) = \|\eta - \lambda\|_Q^2\,,
\end{align*}
where the matrix $Q$ has entries $Q_{T,\tilde T} \defeq \langle T,\tilde T \rangle_{L^2(\rho)}$ for $T,\tilde T \in \cM$.
Here, $Q$ is nothing more than a Gram matrix, which is always positive semi-definite. 
This collection of observations proves the following result.

\begin{theorem}\label{thm: isometry}
Let $\cM$ be a finite family of optimal transport maps {with $Q$ defined as the Gram matrix with entries $Q_{T,\tilde T} = \langle T,\tilde T \rangle_{L^2(\rho)}$ for $T,\tilde T \in \cM$}.
Then, $(\R^{|\cM|}_+,\|\cdot\|_Q)$ is always isometric to $(\coneM_\sharp\rho, \dLOT)$. If, in addition, $\cM$ is a compatible family (i.e., $\coneM_\sharp \rho$ is polyhedral), then $(\R^{|\cM|}_+,\|\cdot\|_Q)$ is isometric to $(\coneM_\sharp\rho, W_2)$.
\end{theorem}

As we develop in the next section,~\Cref{thm: isometry} paves the way for the application of scalable first-order Euclidean optimization algorithms for minimization problems over polyhedral subsets of the Wasserstein space.

\section{Polyhedral optimization in the Wasserstein space}\label{sec: algorithms}

Let $\coneM_\sharp \rho$ be polyhedral and recall the Gram matrix $Q$ from \cref{thm: isometry}, with entries given by $Q_{T,\tilde T} = \langle T, \tilde T \rangle_{L^2(\rho)}$.
We now turn toward the problem of minimizing a functional $\cF$ over $\coneM_\sharp\rho$.
Henceforth, we assume that $Q$ is positive definite, so that $Q^{-1}$ exists.
The positive definiteness of $Q$ follows if the maps $T\in \cM$ are linearly independent in $L^2(\rho)$.

\subsection{Continuous-time gradient flow}

The isometry of~\Cref{sec: isometry} implies that the constrained Wasserstein gradient flow of $\cF$ is equivalent to the gradient flow of the functional $\lambda \mapsto \cF(\mu_\lambda)$ with respect to the $Q$-geometry.\footnote{{See \citet[\S 4.2.1]{nesterov2018lectures} for a thorough discussion on optimization over general Euclidean spaces.}}
The latter gradient flow can be written explicitly as
\begin{align}\label{eq: lbd_flow}
    \dot{\lambda}(t) = -Q^{-1}\,\nabla_\lambda \cF(\mu_{\lambda(t)})\,.
\end{align}
Then, geodesic strong convexity over $\WW$ translates to strong convexity of $\lambda \mapsto \cF(\mu_\lambda)$ over $(\R_+^{|\cM|}, \|\cdot\|_Q)$ for free.
The following theorem\footnote{{In the case where we further constrain the gradient flow to lie in a convex set,~\eqref{eq: lbd_flow} should be replaced by a differential inclusion. Since this is not relevant to the subsequent developments, we omit a fuller discussion of this point.}} establishes convergence rates for this continuous-time flow; see \citet[Appendix D]{lambert2022variational} for a proof.

\begin{theorem}\label{thm: continuous_time}
Suppose $\cF$ is geodesically $m$-strongly convex over $\WW$, for $m \geq 0$. Let $\coneM_\sharp \rho$ be polyhedral.
Then, $\cF$ is geodesically $m$-strongly convex over $\coneM_\sharp \rho$. Moreover, if $\mu_\star \equiv \mu_{\lambda^\star} \in \coneM_\sharp \rho$ is a minimizer of $\cF$ over $\coneM_\sharp \rho$, the following convergence rates hold for the gradient flow~\eqref{eq: lbd_flow}.
\begin{enumerate}
    \item If $m = 0$, then $\cF(\mu_{\lambda(t)}) - \cF(\mu_\star) \leq \frac{1}{2t}\,W_2^2(\mu_{\lambda(0)},\mu_\star)$.
    \item If $m > 0$, then:
    \begin{enumerate}
        \item $W_2^2(\mu_{\lambda(t)},\mu_\star) \leq \exp(-2mt)\,W_2^2(\mu_{\lambda(0)},\mu_\star)$.
        \item $\cF(\mu_{\lambda(t)}) - \cF(\mu_\star) \leq \exp(-2mt)\,(\cF(\mu_{\lambda(0)}) - \cF(\mu_\star))$.
    \end{enumerate}
\end{enumerate}
\end{theorem}

\subsection{Time-discretization made easy}\label{sec:opt_algs}

Appealing to the isometry in~\Cref{sec: isometry}, optimization of a geodesically convex and geodesically smooth functional $\cF$ over a polyhedral set $\coneM_\sharp\rho$ boils down to a finite-dimensional, convex, smooth, \emph{Euclidean} optimization problem of the form 
\begin{align}\label{eq:constrained_opt}
    \min_{\lambda \in \R^{|\cM|}_+} \ \cF(\mu_\lambda)\,.
\end{align}
More generally, we consider optimization over arbitrary convex subsets $K \subseteq \R^{|\cM|}_+$, and we let $\cK \defeq \{T^\lambda \mid \lambda \in K\}$ denote the corresponding subset of $\coneM$.
It leads to the problem
\begin{align*}
    \min_{\lambda \in K} \ \cF(\mu_\lambda)\,.
\end{align*}
Our consideration of general constraint sets $K$ is not purely for the sake of generality, as we in fact use the full power of polyhedral optimization in our application to mean-field VI (in particular, see~\Cref{thm:higher_order} and~\Cref{app:approx_proofs}).

We consider accelerated projected gradient descent \citep{beck2017first}, as well as stochastic projected gradient descent which is useful when only a stochastic gradient is available (as in~\Cref{sec: sgd_mfvi}).
Moreover, when restricted to any \emph{polytope} in the non-negative orthant, we also consider the projection-free Frank--Wolfe algorithm \citep{frank1956algorithm}. We briefly describe the algorithms and state their corresponding convergence guarantees. Note that we could also port over guarantees for other Euclidean optimization algorithms in a similar manner, but we omit them for brevity.

\subsubsection{Accelerated projected gradient descent}\label{sec:gd}

Starting at an initial point $\lambda^{(0)} \in K$, we can solve~\eqref{eq:constrained_opt} by applying a projected variant of Nesterov's accelerated gradient descent method~\citep{Nes1983Accel}, a well-known extrapolation technique that improves upon the convergence rate for projected gradient descent and is optimal for smooth convex optimization~\citep{NemYud1983Complexity}.
The algorithm is given as~\cref{alg: gradient_descent_cj}.
Here, $\proj_{K,Q}(\cdot)$ is the orthogonal projection operator onto $K$ with respect to the $\|\cdot\|_Q$ norm.

We summarize the following well-known convergence results for accelerated projected gradient descent (APGD) below; see \citet[Chapter 10]{beck2017first} for proofs.

\begin{theorem}[Convergence results for APGD]\label{thm: gd_theorem}
    Let $\coneM_\sharp \rho$ be polyhedral and $\cK\subseteq \coneM$ be convex.
    Suppose that $\cF$ is geodesically $m$-strongly convex and $M$-smooth over $\cK_\sharp \rho$ and let $\mu_\star$ denote a minimizer over this set.
    Let $(\lambda^{(t)} : t = 0,1,2,3\ldots)$ denote the iterates of \cref{alg: gradient_descent_cj}.
    \begin{enumerate}
        \item If $m = 0$, then $            \cF(\mu_{\lambda^{(t)}}) - \cF(\mu_\star) \lesssim M t^{-2}\, W_2^2(\mu_{\lambda^{(0)}},\mu_\star)\,.$
        \item If $m > 0$, then for $\kappa \defeq M/m$,
            \begin{enumerate}
            \item $W_2^2(\mu_{\lambda^{(t)}},\mu_\star) \lesssim \kappa\exp(-t/\sqrt{\kappa})\, W_2^2(\mu_{\lambda^{(0)}},\mu_\star)$.
            \item $\cF(\mu_{\lambda^{(t)}}) - \cF(\mu_\star) \leq {(1 - 1/\sqrt{\kappa})}^t\,\bigl( \cF(\mu_{\lambda^{(0)}}) - \cF(\mu_\star) + \tfrac{m}{2}\,W_2^2(\mu_{\lambda^{(0)}},\mu_\star)\bigr)$.
    \end{enumerate}

    \end{enumerate}
\end{theorem}

\begin{algorithm}[t]
\caption{Accelerated projected gradient descent over $\coneM$}\label{alg: gradient_descent_cj}
\begin{algorithmic}
\State \textbf{Input:} $\lambda^{(0)} \in K$, functional $\cF$ ($m$-convex and $M$-smooth in $W_2$), compatible family $\cM$
\State Set $\eta^{(0)} = \lambda^{(0)}$, $\kappa \defeq M/m$ if $m > 0$, and $\gamma_{(0)} = 1$ if $m=0$. 
\For{$t = 0, 1, 2, 3, \ldots$}
\State $\lambda^{(t+1)} \gets \proj_{K,Q}(\eta^{(t)} - \tfrac{1}{M}\,Q^{-1}\,\nabla_\lambda\cF(\mu_{\eta^{(t)}}))$ 
\If{$m > 0$}
\State $\eta^{(t+1)} \gets \lambda^{(t+1)} + \frac{\sqrt{\kappa}-1}{\sqrt{\kappa}+1}\,(\lambda^{(t+1)} - \lambda^{(t)})$ 
\Else
\State $\gamma_{(t+1)} \gets \frac{1 + \sqrt{1 + 4\gamma^2_{(t)}}}{2} $
\vspace{1mm}
\State $\eta^{(t+1)} \gets \lambda^{(t+1)}  + \bigl(\frac{\gamma_{(t)}-1}{\gamma_{(t+1)}}\bigr)\,(\lambda^{(t+1)} - \lambda^{(t)})$
\EndIf
\EndFor
\end{algorithmic}
\end{algorithm}

\subsubsection{Stochastic projected gradient descent}

In some situations, the full gradient $\nabla_\lambda\cF(\mu_{\lambda})$ cannot be computed, usually due to high computational costs. Instead, \emph{stochastic} first-order methods alleviate this issue by instead allowing for the use of an unbiased stochastic gradient oracle, written $\hat\nabla_\lambda {\cF}(\mu_{\lambda})$.\footnote{An unbiased estimator of the gradient is one which $\E_{\mu_{\lambda}}[\hat\nabla_\lambda {\cF}(\mu_{\lambda})] = \nabla_\lambda \cF(\mu_{\lambda})$.} The decreased computational overhead has contributed to the widespread use of stochastic gradient methods as a pillar of modern machine learning~\citep{Bub15CvxOpt}.
We limit our discussions to the case where $\cF$ is smooth and strongly convex, as this setting will be the most relevant later. Other settings readily generalize, though we omit them for brevity.

We provide a description of stochastic projected gradient descent (SPGD) in~\cref{alg: spgd_cj}, and convergence analysis in \cref{thm: sgd_theorem} which requires the following standard assumption on the variance of the unbiased estimator:

\begin{enumerate}[label=$(\mathsf{VB})$]
    \item\label{var_bound} 
    There exist constants $c_0, c_1 \geq 0$ such that for any $\lambda \in K$, the gradient estimate satisfies
    \begin{align*}
        \E[\|Q^{-1}\,(\hat\nabla_\lambda {\cF}(\mu_\lambda) - \nabla_\lambda {\cF}(\mu_\lambda))\|_Q^2] \leq c_0 + c_1\, \E[W_2^2(\mu_{\lambda},\mu_\star)]\,.
    \end{align*}
\end{enumerate}
Note that $c_0,c_1$ in \ref{var_bound} will typically depend on the smoothness and strong convexity parameters of $\cF$, and possibly the dimension of the problem.

\begin{theorem}[Convergence results for SPGD]\label{thm: sgd_theorem}
    Let $\coneM_\sharp \rho$ be polyhedral and $\cK\subseteq \coneM$ be convex.
    Suppose that $\cF$ is geodesically $m$-strongly convex and $M$-smooth over $\cK_\sharp \rho$, let $\mu_\star$ denote a minimizer over this set, and suppose that~\ref{var_bound} holds.
    Let $(\lambda^{(t)} : t = 0,1,2,3\ldots)$ denote the iterates of \cref{alg: spgd_cj} {and let $\eps > 0$ be sufficiently small}. If we choose {step size} $h \asymp \frac{m\eps^2}{c_0} \leq \frac{m}{2c_1} \wedge \frac{1}{2 \kappa M}$, and the number of iterations is at least 
    \begin{align*}
        t \gtrsim \frac{c_0}{m^2\eps^2}\log(W_2(\mu_{\lambda^{(0)}},\mu_\star)/\eps)\,,
    \end{align*}
    then $\E[W_2^2(\mu_{\lambda^{(t)}},\mu_\star)] \leq \eps^2$.
\end{theorem}
For completeness, we provide a short proof of \cref{thm: sgd_theorem} in \cref{app: sgd_general}. 

\begin{algorithm}[t]
\caption{Stochastic projected gradient descent over $\coneM$}\label{alg: spgd_cj}
\begin{algorithmic}
\State \textbf{Input:} $\lambda^{(0)} \in K$, functional $\cF$ ($m$-convex and $M$-smooth in $W_2$), compatible family $\cM$, fixed step-size $h > 0$, and unbiased stochastic gradient oracle $\hat \nabla_\lambda{\cF}(\cdot)$.
\For{$t = 0, 1, 2, 3, \ldots$}
\State $\lambda^{(t+1)} \gets \proj_{K,Q}(\lambda^{(t)} - h\,Q^{-1}\,\hat\nabla_\lambda{\cF}(\mu_{\lambda^{(t)}}))$
\EndFor
\end{algorithmic}
\end{algorithm}

\subsubsection{Frank--Wolfe}

In this section, we consider optimization over a \emph{polytope}, i.e., a set of the form
\begin{align*}
    \convM_\sharp\rho \defeq \Bigl\{\bigl(\textstyle\sum_{T\in\cM} \lambda_T T\bigr)_\sharp \rho \Bigm\vert \lambda \in \Delta_{|\cM|}\Bigr\}\,,
\end{align*}
where $\cM$ is a finite family of compatible maps and $\Delta_{|\cM|}$ denotes the $|\cM|$-dimensional simplex.
Note that $\convM \subseteq \coneM$, where $\coneM_\sharp\rho$ is polyhedral, so that $\convM$ is an example of a convex constraint set $\cK$ considered in the previous subsection.
The convergence guarantees for accelerated projection gradient descent in~\Cref{thm: gd_theorem} therefore apply to optimization over $\convM_\sharp \rho$.

In this setting, however, there is a popular alternative to projected gradient descent known as \emph{conditional gradient descent} or the \emph{Frank--Wolfe} (FW) algorithm \citep{frank1956algorithm,jaggi2013revisiting}. In this scheme, we find a descent direction that ensures our iterates remain within the constraint set. This direction $\eta^{(t)}$ is found at each iterate $\lambda^{(t)}$ by solving the following linear sub-problem:
\begin{align}\label{eq: linear_subproblem_new}
    \eta^{(t)} = \argmin_{\eta \in \Delta_{|\cM|}}{\langle \nabla_\lambda \cF(\mu_{\lambda^{(t)}}), \eta - \lambda^{(t)}\rangle}\,.
\end{align}
Finding this direction can be substantially cheaper than the projection step in~\Cref{alg: gradient_descent_cj}.
Indeed, the sub-problem~\eqref{eq: linear_subproblem_new} does not depend on the matrix $Q$.
It is not hard to see that the minimizer $\eta^{(t)}$ must be attained at one of the $|\cM|$ vertices of the simplex.

The full algorithm is presented in~\cref{alg: frank_wolfe_cj}. Known results provide sublinear convergence of the objective gap, which does not improve under strong convexity assumptions; see \citet[Chapter 13]{beck2017first} for proofs and discussions. 

\begin{theorem}[Convergence results for FW]\label{thm: fw_thm}
Suppose that $\cF$ is geodesically convex and $M$-smooth over $\convM_\sharp\rho$, and let $\mu_\star$ be a minimizer of $\cF$ over this set. Let $(\lambda^{(t)} : t = 0,1,2,3\ldots)$ denote the iterates of \cref{alg: frank_wolfe_cj}, with {step size} $\alpha^{(t)} = 2/(t+2)$. Then,
\begin{align}
    \cF(\mu_{\lambda^{(t)}}) - \cF(\mu_\star) \lesssim M t^{-1}\diam\bigl(\convM_\sharp\rho\bigr)^2\,.
\end{align}
\end{theorem}

\begin{algorithm}[t]
\caption{Frank--Wolfe over $\convM$}\label{alg: frank_wolfe_cj}
\begin{algorithmic}
\State \textbf{Input:} $\lambda^{(0)} \in \Delta_{|\cM|}$, functional $\cF$, and compatible family $\cM$ with $|\cM| = J$.
\For{$t = 0, 1, 2, 3, \ldots$}
\State $j^* \gets \argmin_{j \in [J]}{\langle \nabla_\lambda\cF(\mu_{\lambda^{(t)}}), e_j - \lambda^{(t)}\rangle}$ 
\State $\lambda^{(t+1)} \gets (1-\alpha^{(t)})\, \lambda^{(t)} + \alpha^{(t)}\, e_{j^*} $ \Comment{$\alpha^{(t)}=\tfrac{2}{t+2}$ is a standard step size choice}
\EndFor
\end{algorithmic}
\end{algorithm}

Via the isometry in~\Cref{sec: isometry}, $\diam(\convM_\sharp\rho)$ equals the diameter of $\convM$ in the $Q$-norm.
In terms of the matrix $Q$, this is at most $2\max_{T\in\cM}\sqrt{Q_{T,T}}$.

\begin{remark}
We are not the first to consider applying FW over the Wasserstein space. \cite{kent2021frank} use FW to optimize functionals over the constraint set $\{W_2(\cdot,\pi)\leq \delta\}$ for some $\delta > 0$ and some fixed probability measure $\pi$. In their work, the optimization truly occurs in an infinite-dimensional space. The authors prove various discrete-time rates of convergence under noisy gradient oracles and H{\"o}lder smoothness of the objective function, among other general properties. The core difference between our works is the constraint set of interest, resulting in our algorithm being simpler. Indeed, our setup is purely parametric.
\end{remark}

\subsection{Enriching the family of compatible maps}\label{sec:enrich}

When applying our polyhedral optimization framework to specific problems of interest, it is sometimes useful to first enrich the compatible family. For example, one notable advantage of doing so is that it increases the expressive power of the constraint set.
Another example is that for our application to mean-field VI in~\Cref{sec: mfvi_main}, it will be necessary for us to ensure a uniform lower bound on the Jacobian derivatives of the maps in our family (i.e., they are gradients of \emph{strongly convex} potentials).

The second issue can be addressed by adding $\alpha \id$ to each member of the family.
Indeed, by~\Cref{lem:cones} and~\Cref{lem:add_id}, $\cone(\cM \cup \{\id\})$ is a compatible family, and then we can restrict to the convex subset $\cK \subseteq \cone(\cM \cup \{\id\})$ corresponding to $\lambda$ for which the coefficient $\lambda_{\id}$ in front of $\id$ is $\alpha$. 
The guarantees of~\Cref{sec:gd} then apply directly to optimization over $\cK_\sharp \rho$.
However, we prefer to handle the $\alpha\id$ term separately, and so we define the \emph{cone generated by $\cM$ with tip $\alpha\id$} to be the family
\begin{align*}
    \cone(\cM;\, \alpha \id)
    &\defeq \alpha \id + \coneM\,.
\end{align*}

Similarly, to address the first issue, we would like to enrich our compatible family by adding translations, via~\Cref{lem:translations}. To this end, we define our  \emph{augmented cone}, $\augconeM$ for short, to be
\begin{align*}
    \augconeM
    &\defeq \Bigl\{\textstyle\sum_{T\in\cM} \lambda_T T + v \Bigm\vert \lambda \in \R_+^{|\cM|}, \; v\in\R^d\Bigr\}\,.
\end{align*}
Similarly, we define
\begin{align*}
    \augtip
    &\defeq \alpha \,{\id} + \augconeM\,.
\end{align*}
The augmented cone is parameterized by $(\lambda,v) \in \R_+^{|\cM|} \times \R^d$.
We may assume that each of the maps $T\in \cM$ has mean zero under $\rho$, since this does not affect the augmented cone.
Under this assumption, it is easy to see (c.f.\ the proof of~\autoref{prop: kl_smooth}) that we still obtain an isometry with a Euclidean metric: $W_2^2(\mu_{\eta,u}, \mu_{\lambda,v}) = \|\eta-\lambda\|_Q^2 + \|u-v\|^2$.
In this setting, the first-order algorithms must be modified to compute the gradient and projection steps with respect to this metric.

\begin{remark} [Broader impact of our framework]
{{We now pause to briefly discuss the broader impact of polyhedral sets. We want to stress that, even without compatibility, our framework can be used to optimize functionals over any convex subset of the tangent space, provided that the functional is convex with respect to the \emph{linearized} optimal transport distance.
In turn, this is equivalent to requiring that the functional is convex along generalized geodesics~\citep[see][\S 9.2]{ambrosio2005gradient}, which is typically the case when the functional is convex in the Wasserstein geometry; for example, it holds for the KL divergence with respect to a log-concave measure.
This substantially expands the scope of applications as it allows for optimization over \emph{any} convex subset of the tangent space, not just compatible ones.}}
\end{remark}

\section{Application to mean-field variational inference}\label{sec: mfvi_main}

As our main application of polyhedral optimization over the Wasserstein space, we turn to variational inference (VI) \citep{blei2017variational}. In this framework, we are given access to an unnormalized probability measure, known as the posterior, written $\pi\propto \exp(-V)$, from which we wish to obtain samples for downstream tasks. In principle, one can draw samples from $\pi$ via Markov chain Monte Carlo methods, but these have computational drawbacks, such as potentially long burn-in times. Instead, VI suggests to minimize the Kullback--Leibler (KL) divergence over a constraint set to obtain a proxy measure that is easy to sample from. Commonly used constraint sets in the literature include the space of non-degenerate Gaussians, location-scale families, mixtures of Gaussians, and the space of product measures. 

For a general constraint set $\cC$, the VI optimization problem reads
\begin{align}\label{eq: vi_general}
    \pi^\star_{\cC} \in \argmin_{\mu \in \cC}\kl{\mu}{\pi} \defeq \argmin_{\mu \in \cC}  \int V \dd \mu + \int \log \mu \dd \mu + \log Z\,,
\end{align}
where $Z$, the unknown normalizing constant of $\pi \propto \exp(-V)$ plays no part in the optimization problem. The following assumption, which will play a crucial role in our analyses in \cref{sec: mfvi_approx} and \cref{sec: mfvi_convergence}, is standard in the literature on log-concave sampling~\citep{ChewiBook}:
\begin{enumerate}[label=$(\mathsf{WC})$]
    \item\label{well_cond} $\pi$ is $\ell_V$-strongly log-concave and $L_V$-log-smooth, i.e., $\ell_V I \preceq \nabla^2 V \preceq L_V I$ for $\ell_V, L_V > 0$.
\end{enumerate}
In brief, we say that $\pi$ is \emph{well-conditioned}.
We denote by $\kappa \defeq L_V/\ell_V$ the condition number.

The following lemma allows us to refer to the unique minimizer of the VI problem, which follows from the strong geodesic convexity of the KL divergence (see the discussions around \cref{prop: kl_convex}).

\begin{lemma}\label{lem: kl_sc}
    Suppose $\cC$ is a geodesically convex subset of $\cP_2(\R^d)$, and suppose that $\pi$ is strongly log-concave. Then, there is a unique minimizer of $\kl{\cdot}{\pi}$ over $\cC$.
\end{lemma}

Despite the widespread use of variational inference in numerous settings~\citep[see for example][]{wainwright2008graphical, blei2017variational}, explicit guarantees have only recently been established for a few constraint families. Recently, \cite{lambert2022variational,diao2023forward} obtained computational guarantees for Gaussian VI by way of constrained Wasserstein gradient flows.
\cite{domke2020provable, domke2023provable, kim2023black} considered VI for location-scale families and provided algorithmic guarantees, though they abstained from the gradient flow formalism. Subsequent work by \citet{yi2023bridging} made this connection precise.

In the sequel, we develop end-to-end computational guarantees for mean-field variational inference. This is done in five stages: 
\begin{enumerate}
    \item Transfer assumptions on the posterior $\pi$, namely \ref{well_cond}, to the mean-field solution $\pi^\star$ (see \cref{prop: mfe_properties} in \Cref{sec: mfvi}).
    \item Use the properties of $\pi^\star$ to obtain regularity properties of the optimal transport map $T^\star$ from the standard Gaussian measure to $\pi^\star$, via Caffarelli's contraction theorem and the Monge{--}Amp\`ere equation (see~\Cref{thm:regularity} in~\Cref{sec:regularity}).
    \item Show that polyhedral sets in the Wasserstein space can approximate mean-field measures arbitrarily well, making use of the regularity properties of the optimal transport map, approximation theory, and Wasserstein calculus (see \cref{thm: pi_closeness}, \cref{thm:higher_order}, and~\Cref{thm:approx_improved} in \Cref{sec: mfvi_approx}).
    \item Provide convergence guarantees for optimizing the KL divergence over these polyhedral sets (see \cref{thm: acc_vi} in \Cref{sec: mfvi_convergence}).
    \item Describe implementation details for our final algorithm (see~\Cref{sec: implementation_details}).
\end{enumerate}

\subsection{Mean-field variational inference}\label{sec: mfvi}
In mean-field VI, the constraint set is the space of product measures over $\Rd$, written $\cP(\R)^{\otimes d}$. Thus, the optimization problem is
\begin{align}\label{eq: obj_mfvi}
    \pi^\star \in \argmin_{\mu \in \cP(\R)^{\otimes d}}\kl{\mu}{\pi}\,,
\end{align}
where, by design, the constraint set forces the minimizers to be of the form
\begin{align}\label{eq: mfvi_min}
    \pi^\star(x_1,\ldots,x_d) = \bigotimes_{i=1}^d \pi^\star_i(x_i)\,.
\end{align}

Mean-field VI has a rich history in the realm of statistical inference; see Section 2.3 in \cite{blei2017variational} for a brief historical introduction. Despite being widely used, computational and statistical guarantees have only recently emerged. A standard algorithm to solve \eqref{eq: obj_mfvi} is Coordinate Ascent VI (CAVI)~\citep[see][Section 2.4]{blei2017variational}, the updates for which can be implemented for certain conjugate models.
Guarantees for CAVI were provided recently in~\citet{bhattacharya2023convergence} under a generalized correlation condition for $\pi$; {see also \cite{arnese2024convergence} and \cite{lavenant2024convergence}}.

More closely related to our work is the use of Wasserstein gradient flows.
The work of \cite{lacker2023independent} connects mean-field VI to constrained Wasserstein gradient flows, providing continuous-time guarantees via projected log-Sobolev inequalities but without a concrete algorithmic implementation; see also \cite{lacker2024mean}.
Also, in the context of a Bayesian latent variable models, convergence guarantees for a Wasserstein gradient flow under a well-conditioned assumption at the population level was established by~\citet{yao2022mean}.
Toward the issue of implementation, they suggested two strategies based on particle approximation combined with either Langevin sampling or optimization over transport maps respectively, but they did not analyze the error arising from the particle approximation.
\cite{zhang2020theoretical} study the theoretical and computational properties of mean-field variational inference in the context of community detection. Despite the promising nature of these works, implementation remains a challenge in complete generality.
\looseness-1

\paragraph*{Mean-field equations.} Via calculus of variations, one can readily derive the following system of \emph{mean-field equations} from~\eqref{eq: obj_mfvi}: for $i \in [d]$,
\begin{align}\label{eq: mf_equations}
    \pi^\star_i(x_i) \propto  \exp\Bigl(- \int_{\R^{d-1}}V(x_1,\ldots,x_d)\, \bigotimes_{j\neq i} \pi^\star_j(\!\dd x_j)\Bigr)\,.
\end{align}
These are also sometimes called \emph{self-consistency equations}; we give a derivation in~\Cref{app:mf_eq_pf}.
From the structure of $\pi^\star$, we can prove the following result.

\begin{proposition}\label{prop: mfe_properties}
Suppose that $\pi$ is well-conditioned~\ref{well_cond}. Then,~\eqref{eq: obj_mfvi} admits a \emph{unique} minimizer of the form \eqref{eq: mfvi_min}, where each $\pi^\star_i$ is well-conditioned~\ref{well_cond} with the same parameters $\ell_V$, $L_V$ as $\pi$.
\end{proposition}

Uniqueness of the minimizer follows as a corollary of \cref{lem: kl_sc} and \citet[Proposition 3.2]{lacker2023independent}, which shows that $\cP(\R)^{\otimes d}$ is a geodesically convex subset of the Wasserstein space, and the individual $\pi^\star_i$ measures being well-conditioned is immediate from~\eqref{eq: mf_equations}.

\paragraph*{Our approach.}
We approach solving mean-field VI by optimizing over a suitably rich family of compatible maps. To this end, we want to relate~\eqref{eq: obj_mfvi} to 
\begin{align}\label{eq: vi_cm}
    \pi_\diamond^\star \defeq (T_\diamond^\star)_\sharp\rho \in \argmin_{\mu\in\Pdiam} \kl{\mu}{\pi}\,,
\end{align}
where $\Pdiam$ is a polyhedral subset of the Wasserstein space (\Cref{defn:polyhedral}).
Recall that polyhedral subsets of the Wasserstein space are geodesically convex (see~\Cref{thm: comp_geodconv}).
Combined with \cref{lem: kl_sc}, the following corollary is immediate.

\begin{corollary}
    Suppose that $\Pdiam$ is a polyhedral subset of the Wasserstein space, and that $\pi$ is well-conditioned~\ref{well_cond}. Then the minimizer to~\eqref{eq: vi_cm} is \emph{unique}, denoted by $\pi_\diamond^\star$.
\end{corollary}

Borrowing inspiration from the existing literature combining normalizing flows and optimal transport~\citep{chen2018neural, finlay2020learning, finlay2020train, Huaetal21CvxFlows}, our goal is to transfer the difficulty of estimating the measure $\pi^\star$ to estimating an appropriate optimal transport map. Indeed, $\Pdiam$ is a collection of pushforwards of a base measure $\rho$ via optimal transport maps. In this work, we will provide a systematic way of choosing both $\rho$, the base measure, and $\cM$, the set of optimal transport maps which generates $\Pdiam$.

A natural candidate for the base density $\rho$ is the standard Gaussian distribution in $\R^d$.
Beyond its naturality, this choice is justified by powerful
regularity results, described in the next section, for the optimal transport map $T^\star$ from $\rho$ to the mean-field solution $\pi^\star$.
This regularity result, in turn, will feed into the approximation theory of~\Cref{sec: mfvi_approx}.

\subsection{Regularity of optimal transport maps between well-cond\-itioned product measures}\label{sec:regularity}

In this section, we study the regularity of the optimal transport map from the standard Gaussian to the mean-field solution $\pi^\star$.
More generally, our regularity bounds hold for the optimal transport map from the Gaussian to any well-conditioned product measure, or between any two well-conditioned product measures $\mu$ and $\nu$ (either by writing $T^{\mu\to\nu}$ as $T^{\rho \to \nu} \circ (T^{\rho\to\mu})^{-1}$ and directly applying the results of this section, or by repeating the arguments thereof).

\begin{theorem}\label{thm:regularity}
    Let $\rho = \cN(0,I)$ and suppose that $\pi$ is well-conditioned~\ref{well_cond}. Then, there exists a unique, coordinate-wise separable optimal transport map from $\rho$ to $\pi^\star$, the minimizer to \eqref{eq: obj_mfvi}, written $T^\star(x) = (T^\star_1(x_1),\ldots,T^\star_d(x_d))$. Each map $T^\star_i$ satisfies
    \begin{align*}
        \sqrt{1/L_V} \leq (T^\star_i)'\leq \sqrt{1/\ell_V}\,.
    \end{align*}
    Moreover, we have the higher-order regularity bounds
    \begin{align}\label{eq:higher_regularity}
        |(T^\star_i)''(x)|
        &\lesssim \frac{\kappa}{\sqrt{\ell_V}}\,(1+|x|)\,,\qquad\text{and}\qquad
        |(T^\star_i)'''(x)|
        \lesssim \frac{\kappa^2}{\sqrt{\ell_V}}\,(1 + |x|^2)\,.
    \end{align}
\end{theorem}

The bounds on $(T^\star_i)'$ in \Cref{thm:regularity} in fact follow immediately from two landmark results in optimal transport, and \cref{prop: mfe_properties}. First, since $\rho$ admits a density, then Brenier's theorem \citep{Bre91} states that there always exists a unique optimal transport map from $\rho$ to any target measure, in this case, $\pi^\star$. Obviously, since both $\rho$ and $\pi^\star$ are product measures, the corresponding optimal transport map is coordinate-wise separable. Then, Caffarelli's contraction theorem \citep{caffarelli2000monotonicity} yields tight lower and upper bounds on the derivatives of each component of $T^\star$ as a function of the strong log-concavity and log-smoothness parameters of $\rho$ and $\pi^\star$. See, e.g.,~\citet[Theorem 4]{ChePoo23Caffarelli} for a precise statement of the contraction theorem, and a short proof based on entropic optimal transport.

On the other hand, we have not seen the bounds~\eqref{eq:higher_regularity} in the literature.
In general, regularity theory for optimal transport is notoriously challenging due to the fully non-linear nature of the associated Monge{--}Amp\`ere PDE\@; see~\citet[Section 4.2.2]{villani2021topics} for an exposition to Caffarelli's celebrated regularity theory.
Here, we can avoid difficult arguments by exploiting the coordinate-wise separability of the transport map and straightforward computations with the Monge{--}Amp{\`e}re equation.
See~\Cref{app:regularity_proofs} for the proof.

The regularity we obtain is essentially optimal, since we started with information on the derivatives of $\pi^\star$ up to order two, and we obtain regularity bounds for the Kantorovich potential (of which $T^\star$ is the gradient) up to order \emph{four}.
Such higher-order regularity bounds are not only useful for obtaining sharper approximation results, but are in fact essential for establishing the key result~\Cref{thm:approx_improved} in~\Cref{sec: mfvi_approx}.

\begin{remark}
In prior works that statistically estimate optimal transport maps on the basis of samples (such as~\citet{deb2021rates,hutter2021minimax,manole2021plugin,pooladian2021entropic,divol2022optimal}), bounds on the Jacobian of the optimal transport map of interest are necessary and standard. In contrast, here these bounds hold as a consequence of our problem setting (in particular, from~\ref{well_cond}). 
\end{remark}

\subsection{Approximating the mean-field solution with compatible maps}\label{sec: mfvi_approx}

So,~\Cref{thm:regularity} tells us that we can view $\pi^\star = (T^\star)_\sharp\rho$, where $T^\star$ obeys desirable regularity properties.
The goal of this section is to demonstrate that we can prescribe a class of maps $\cM$ such that the minimizer of the KL divergence over $\Pdiam \defeq \augtip_\sharp \rho$,
\begin{align*}
    \pi^\star_{\diamond} \in \argmin_{\mu\in\Pdiam}\kl{\mu}{\pi}\,,
\end{align*}
is close to $\pi^\star$ in the Wasserstein distance.
Then,~\Cref{sec: mfvi_convergence} will provide guarantees for computing $\pi^\star_\diamond$ via KL minimization over this set.

The first step is to prove an approximation theorem: there \emph{exists} an element $\hat\pi_\diamond \in \Pdiam$ such that $\pi^\star$ is close to $\hat\pi_\diamond$. We state this as the following general result.
Here, we write $\|D(\bar T- \hat T)\|_{L^2(\rho)}^2$ for the quantity $\int \|D(\bar T - \hat T)\|_{\rm F}^2\dd \rho$.

\begin{theorem}\label{thm: pi_closeness}
    Let $\rho = \cN(0, I)$.
    For any $\eps > 0$, there exists a compatible family $\cM$ of optimal transport maps of size $\widetilde O(\kappa^{1/2} d^{5/4}/\varepsilon^{1/2})$,
    with the following property.
    For any coordinate-wise separable map $\bar T : \R^d\to\R^d$ with Jacobian satisfying the first and second derivative bounds of~\Cref{thm:regularity}, there exists $\hat T \in \augtip$, with $\alpha = 1/\sqrt{L_V}$, such that $W_2( \bar T_\sharp \rho, \hat T_\sharp \rho) = \|\bar T - \hat T\|_{L^2(\rho)} \le \eps/\ell_V^{1/2}$ and $\|D(\bar T - \hat T)\|_{L^2(\rho)} \lesssim \kappa^{1/2} d^{1/4} \varepsilon^{1/2}/\ell_V^{1/2}$.
\end{theorem}

Approximation theory has a large literature which aims at proving uniform rates of approximation over various function classes by linear combinations of well-chosen basis elements. Typical choices of basis functions include polynomials, splines, wavelets, etc., with more recent literature investigating approximations via neural networks. 

While resting on standard techniques, the most important departure of our result from the literature is the coordinate-wise structure of $\bar T$, which allows for approximation rates that do not incur the curse of the dimensionality, in the sense that the cardinality of $|\cM|$ does not depend exponentially on the dimension $d$.
Observe the presence of a structural constraint: in one dimension, the problem essentially boils down to approximating a monotonically increasing function via conic combinations of the generating set $\cM$.

Our construction is described as follows. Let $R > 0$ denote a truncation parameter, and let $\delta > 0$ denote a mesh size. We partition the interval $[-R, +R]$ into sub-intervals of size $\delta$.
Then, $\cM$ consists of all functions of the form $x\mapsto (0,\dotsc,0, \psi(\delta^{-1}\,(x_i-a)), 0,\dotsc,0)$, where only the $i$-th coordinate of the output is non-zero, $I = [a,a+\delta]$ is a sub-interval of size $\delta$, and $\psi : \R\to\R$ is piecewise linear, defined via $\psi(x) \defeq 1 \wedge x_+$.
Proofs are given in~\Cref{app:approx_proofs}.

This piecewise linear construction exploits the smoothness of $\bar T$ up to order two, but no further.
On the other hand, from~\Cref{thm:regularity} we see that $T^\star$ also obeys a bound on its \emph{third} derivative, so we can expect to obtain better approximation rates through a smoother dictionary.
This is indeed the case, but the approximating set becomes more complicated (in particular, it is no longer the pushforward of a pointed cone, but a general polyhedral set), so we defer the details to~\Cref{app:approx_proofs}.

\begin{theorem}\label{thm:higher_order}
    There exists a polyhedral set $\Pdiam$ with an explicit generating family (see~\Cref{app:approx_proofs}) of size $\widetilde O(\kappa^{2/3} d^{7/6}/\varepsilon^{1/3})$ with the following property.
    In the setting of~\Cref{thm: pi_closeness}, assume also that each component $\bar T_i$ of $\bar T$ obeys the third derivative bound in~\Cref{thm:regularity}.
    Then, there exists $\hat T \in \Pdiam$ such that $W_2(\bar T_\sharp \rho, \hat T_\sharp \rho) = \|\bar T - \hat T\|_{L^2(\rho)} \le \varepsilon/\ell_V^{1/2}$ and $\|D(\bar T - \hat T)\|_{L^2(\rho)} \lesssim \kappa^{2/3} d^{1/6} \varepsilon^{2/3}/\ell_V^{1/2}$.
\end{theorem}

{
\begin{remark}
We note that the worse dependence on $\kappa$ for the smoother dictionary is 
due to our derivative bounds for the optimal transport map; we have no reason to believe it is fundamental.
\end{remark}
}
The two preceding results show that we can, with prior knowledge of $\pi^\star$, \emph{construct} some $\hat\pi_\diamond \in \Pdiam$ which is close to $\pi^\star$, but it does not guarantee that we can \emph{find} $\hat\pi_\diamond$ easily.
The next result addresses this issue by showing that $\pi^\star$ is close to the \emph{minimizer} $\pi^\star_\diamond$ of the KL divergence over $\Pdiam$, and hence can be computed using the algorithms in~\Cref{sec: mfvi_convergence}. As the proof reveals, establishing this statement is related to a geodesic \emph{smoothness} property for the KL divergence, which is quite non-trivial since the entropy is non-smooth over the full Wasserstein space (see the further discussion in the next section).
We are able to verify this smoothness property on the geodesic connecting $\hat\pi_\diamond$ to $\pi^\star$ using the bounds on $\|D(\hat T_\diamond - T^\star)\|_{L^2(\rho)}$ in our approximation results (\Cref{thm: pi_closeness} and~\Cref{thm:higher_order}). The proof of~\Cref{thm:approx_improved} is also found in~\Cref{app:approx_proofs}.

\begin{theorem}\label{thm:approx_improved}
    The mean-field solution $\pi^\star$ is close to the \emph{minimizer} $\pi^\star_\diamond$ of the KL divergence over $\Pdiam$ with corresponding generating family $\cM$, in the sense that $\sqrt{\ell_V}\,W_2(\pi^\star_\diamond,\pi^\star)\le\eps$, in the following two cases.
    \begin{enumerate}
        \item For the piecewise linear construction of~\Cref{thm: pi_closeness}, the size of the family is bounded by $|\cM| \le \widetilde O(\kappa^2 d^{3/2}/\eps)$.
        \item For the higher-order approximation scheme of~\Cref{thm:higher_order}, the size of the family is bounded by $|\cM| \le \widetilde O(\kappa^{3/2} d^{5/4}/\eps^{1/2})$.
    \end{enumerate}
\end{theorem}

\subsection{Computational guarantees for mean-field VI}\label{sec: mfvi_convergence}

Having identified polyhedral subsets $\Pdiam$ of the Wasserstein space over which the KL minimizer $\pi^\star_\diamond$ is close to the desired mean-field VI solution $\pi^\star$, we are now in a position to apply our theory of polyhedral optimization and thereby obtain novel computational guarantees for mean-field VI\@.
Recall that $\pi \propto \exp(-V)$, and 
\begin{align}\label{eq: kl_val}
    \kl{\mu}{\pi} = \cV(\mu) + \cH(\mu) \defeq \int V \dd \mu + \int \log \mu\dd\mu +  \log Z\,,
\end{align}
where $Z > 0$ is the normalizing constant of $\pi$.

For concreteness, we focus our discussion on the setting in which $\Pdiam = \augtip_\sharp\rho$, such as in~\Cref{thm: pi_closeness}, although the discussion below can be adapted to more general polyhedral sets.
As in~\Cref{sec:opt_algs}, we apply Euclidean optimization algorithms over the parameterization of $\augtip$; see~\Cref{sec: implementation_details} for a discussion of implementation.

In order to apply the algorithmic guarantees from \cref{sec:opt_algs}, we must verify the strong geodesic convexity and geodesic smoothness of the KL divergence over the set $\Pdiam$.
Strong convexity follows from the celebrated fact that the KL divergence with respect to an $\ell_V$-strongly log-concave measure $\pi$ is $\ell_V$-strongly geodesically convex~\cite[see][Particular Case 23.15]{Vil08}, together with the geodesic convexity of $\Pdiam$ (\Cref{thm: comp_geodconv} and~\cref{sec:enrich}).

\begin{proposition}[Strong convexity of the KL divergence over geodesically convex sets]\label{prop: kl_convex}
    Assume that $\pi$ is well-conditioned~\ref{well_cond}.
    Then, the KL divergence $\kl{\cdot}{\pi}$ is $\ell_V$-strongly geodesically convex over any geodesically convex subset of the Wasserstein space.
\end{proposition}

Smoothness of the KL divergence, however, is more subtle, owing to the non-smoothness of the entropy $\cH$ over the full Wasserstein space; see~\citet{diao2023forward} for further discussion of this point.
Prior works therefore established smoothness over restricted subsets of the Wasserstein space~\citep[e.g.,][]{lambert2022variational}, or utilized proximal methods which succeed in the absence of smoothness~\citep[e.g.,][]{diao2023forward}. We adopt the former approach, and for this we require a further property of the family $\cM$ of generating maps.

First, without loss of generality, we may assume that each $T\in\cM$ has mean zero under $\rho$: $\int T \dd\rho = 0$.
Indeed, subtracting the means from the maps in the generating set does not affect $\augtip$, since $\augtip$ is augmented by translations.
Assuming now that $\cM$ is centered, we recall the Gram matrix $Q$ with entries $Q_{T,\tilde T} \defeq \langle T,\tilde T\rangle_{L^2(\rho)}$.
We also form the Gram matrix of the Jacobians, $Q^{(1)}$, with entries $Q^{(1)}_{T,\tilde T} \defeq \langle DT, D\tilde T\rangle_{L^2(\rho)} \defeq \int \langle DT, D\tilde T\rangle\dd\rho$.
Our main assumption on $\cM$ is an upper bound on $Q^{(1)}$ in terms of $Q$. We refer to families $\cM$ satisfying this condition as \emph{regular}.

\begin{enumerate}[label=$(\Upsilon)$]
    \item\label{regular_dict} There exists $\Upsilon > 0$ such that for the Gram matrices associated with a centered family $\cM$, it holds that $Q^{(1)} \preceq \Upsilon Q$.
\end{enumerate}
We remark that when $\cM$ is constructed as a direct sum of univariate families via~\Cref{lem: direct_sum} (c.f.~\Cref{rmk:generator_of_direct_sum}), as in our approximation results (\Cref{sec: mfvi_approx}), the matrices $Q$ and $\tilde Q$ have a $d\times d$ block diagonal structure; see~\Cref{sec: implementation_details} for details.
Consequently, the regularity $\Upsilon$ of the family $\cM$ is the same as the regularity parameter for the \emph{univariate} family used to construct $\cM$, and is therefore nominally ``dimension-free''.\footnote{However, if one wishes to maintain the same quality of approximation in high dimension, our approximation results in~\Cref{sec: mfvi_approx} require taking the size of the univariate family to scale mildly with the dimension, and in this case the parameter $\Upsilon$ may indeed scale with the dimension.}

We can now establish our geodesic smoothness result for the KL divergence over the augmented and pointed cone $\augtip_\sharp \rho$, where $\cM$ is a regular generating family.

\begin{proposition}[Smoothness of the KL divergence over $\augtip_\sharp \rho$]\label{prop: kl_smooth}
    Assume that $\pi$ is well-conditioned~\ref{well_cond} and that $\cM$ is regular~\ref{regular_dict}.
    Then, $\kl{\cdot}{\pi}$ is $M$-geodesically smooth over $\augtip_\sharp \rho$, with smoothness constant bounded by
\begin{align*}
    M \le L_V + \Upsilon/\alpha^2\,.
\end{align*}
\end{proposition}

From~\Cref{thm:regularity}, we know that the optimal transport map $T^\star$ from $\rho$ to the mean-field solution $\pi^\star$ is the gradient of a $1/\sqrt{L_V}$-strongly convex potential, so we take $\alpha = 1/\sqrt{L_V}$.
The smoothness constant for the KL divergence then becomes $(1+\Upsilon)\,L_V$.
With these results in hand, we can state our accelerated convergence guarantees for mean-field VI, which follow directly from the previous propositions, and \cref{thm: gd_theorem}.

\begin{theorem}[Accelerated mean-field VI]\label{thm: acc_vi}
    Assume that $\pi$ is well-conditioned~\ref{well_cond} and that $\cM$ is regular~\ref{regular_dict}.
    Let $\pi^\star_\diamond$ denote the unique minimizer of $\kl{\cdot}{\pi}$ over the polyhedral set $\Pdiam = \augtip_\sharp \rho$ with $\alpha = 1/\sqrt{L_V}$. Then, the iterates of accelerated projected gradient descent yield a measure $\mu_{(t)}$ with the guarantee $W_2(\mu_{(t)}, \pi^\star_{\diamond}) \leq \eps$, with a number of iterations bounded by
    \begin{align*}
        t = O\Bigl(\sqrt{\kappa\,(1 + \Upsilon)}\,\log\bigl(\sqrt{\kappa}\,W_2(\mu_{(0)},\pi^\star_{\diamond})/\eps\bigr)\Bigr)\,,
    \end{align*}
    where $\kappa \defeq L_V/\ell_V$ is the condition number of $\pi$. 
\end{theorem}

By combining~\Cref{thm: acc_vi} with our approximation result in~\Cref{thm:approx_improved}, which provides a bound on $W_2(\pi^\star_\diamond,\pi^\star)$ for explicit choices of $\Pdiam$ with corresponding bounds on the size of $|\cM|$, we can then ensure that the iterate $\mu_{(t)}$ is close to $\pi^\star$ in the Wasserstein distance.
{Namely, we ensure that $W_2(\mu_{(t)}, \pi^\star) \le \varepsilon$ provided that we use either of the dictionaries in~\Cref{thm:approx_improved} and we take the number of iterations $t$ as in~\Cref{thm: acc_vi}; here, $t$ is the iteration complexity, whereas the bounds on the size of the dictionary in~\Cref{thm:approx_improved} govern the per-iteration cost.}

{As previously mentioned, our analysis is made possible by bypassing the non-differentiability of entropy over $\cP(\R)$ and instead optimizing over pointed cones characterized by a univariate compatible family $\cM$. Thus, the constant $\Upsilon > 0$ should blow up as the polyhedral set approaches $\cP(\R)$. This fact is summarized in the following lemma for the piecewise-linear family. 
\begin{lemma}\label{lem:Upsilon_lemma}
Let $\cM$ be the piecewise linear construction of Theorem~\ref{thm: pi_closeness}. Then, $\Upsilon \lesssim |\cM_1|^2$, where $\cM_1$ is the generating family in a single dimension. Since $|\cM_1| = J$, the bound is equivalently $\Upsilon \lesssim J^2$.
\end{lemma}}
{As a corollary, we can fully characterize the runtime of solving the MFVI problem. 
\begin{corollary}[End-to-end guarantees for MFVI] Consider the setting of Theorem~\ref{thm: acc_vi}. Then the required runtime to compute $\pi_\diamond^\star$ becomes 
\begin{align*}
    t = O(J\kappa^{1/2}\log(\sqrt{\kappa}d/\eps))\,.
\end{align*}
If $\pi_\diamond^\star$ is meant to approximation $\pi^\star$, then we can use the approximation guarantees from Theorem~\ref{thm:approx_improved} to obtain the complete convergence guarantee of
\begin{align*}
    t = O({\kappa^{5/2}}d^{1/2}\eps^{-1}\log(\sqrt{\kappa}d/\eps))\,.
\end{align*}
\end{corollary}
As we demonstrate below, the regime of $J=O(1)$ appears to suffice numerically. To the best of our knowledge, this constitutes the first \emph{accelerated} and \emph{end-to-end} convergence result for mean-field VI\@.
See~\Cref{sec: mfvi} for comparisons with the literature.
}
\subsection{Algorithms for mean-field VI}\label{sec: algorithm_and_sgd}

In this section, we discuss implementation details for our proposed mean-field VI algorithm, which includes an analysis of stochastic gradient descent over our polyhedral sets.

\subsubsection{Implementation details}\label{sec: implementation_details}

Recall that the goal is to compute a product measure approximation to $\pi$ which has density proportional to $\exp(-V)$ on $\R^d$.

\paragraph*{Building the family of maps.}
The first step is to build a family $\cM_1$ of increasing maps $\R\to\R$.
The specification of these maps is left to the user; in~\Cref{sec: mfvi_approx}, we have provided an example of a family of maps with favorable approximation properties.
For later purposes, it is also important to center the maps to ensure that they have mean zero under $\rho$; this is done by computing the expectations of the maps via one-dimensional Gaussian quadrature and subtracting the means.

Let $J$ denote the size of $|\cM_1|$ and write $\cM_1 = \{T_1,\dotsc,T_J\}$.

\paragraph*{Parameterization of the cone.}
As discussed in~\Cref{sec:enrich}, it is useful to augment the cone with translations.
Once the one-dimensional family $\cM_1$ has been specified, it generates the $d$-dimensional augmented cone of maps parameterized by $(\lambda,v) \in \R_+^{Jd}\times\R^d$: the corresponding map $T^{\lambda,v}$ is given by $T^{\lambda,v}(x) = \alpha x + \sum_{i=1}^d \sum_{j=1}^J \lambda_{i,j} T_j(x_i)\,e_i + v$. 

\paragraph*{Construction of the $Q$ matrix.}
For concreteness, let us fix the reference measure $\rho$ to be the standard Gaussian $\cN(0,I_d)$.
We must compute the $Jd\times Jd$ matrix $Q$, with entries $Q_{(i,j); (i',j')} \defeq \int \langle T_j(x_i)\,e_i, T_{j'}(x_{i'})\,e_{i'}\rangle \,\rho(\!\dd x)$.
From this expression, it is clear that $Q$ is block diagonal; in fact, if we let $Q^{\cM_1}$ denote the matrix corresponding to the one-dimensional family, with entries $Q^{\cM_1}_{j,j'} \defeq \int T_j T_{j'}\dd\rho_1$ (here $\rho_1$ is the one-dimensional standard Gaussian), then $Q = I_d \otimes Q^{\cM_1}$, and hence the full matrix $Q$ never has to be stored in memory.

The entries of the $J\times J$ matrix $Q^{\cM_1}$ can be precomputed, either via Monte Carlo sampling from $\rho_1$, or via one-dimensional Gaussian quadrature.

\paragraph*{Computation of the gradient and projection.}
In order to apply the algorithms in~\Cref{sec: algorithms}, we must specify the gradient of $\kl{(T^{\lambda,v})_\sharp \rho}{\pi}$ w.r.t.\ $(\lambda,v)$ and the projection operator w.r.t.\ the $Q$-norm, $\|\cdot\|_Q$. Recall that we compute the gradients and projections for the $\lambda$ variable w.r.t.\ $\|\cdot\|_Q$, and for the $v$ variable in the standard Euclidean norm.

Using the change of variables formula,
\begin{align*}
    \kl{(T^{\lambda,v})_\sharp \rho}{\pi}
    &= \int [V(T^{\lambda,v}(x)) - \log \det DT^{\lambda}(x)] \,\rho(\!\dd x) + \int \log \rho \dd \rho + \log Z\,.
\end{align*}
The partial derivatives are therefore computed to be
\begin{align}\label{eq: grad_kl}
    \begin{aligned}
        \partial_{\lambda_{i,j}} \kl{(T^{\lambda,v})_\sharp \rho}{\pi}
        &= \int \bigl[\partial_i V(T^{\lambda,v}(x))\, T_j(x_i) - \langle e_i, (DT^\lambda)^{-1}(x)\,e_i\rangle\,T_j'(x_i) \bigr]\, \rho(\!\dd x)\,, \\
        \nabla_v \kl{(T^{\lambda,v})_\sharp \rho}{\pi}
        &= \int \nabla V(T^{\lambda,v}(x))\,\rho(\!\dd x)\,.
    \end{aligned}
\end{align}
For the terms explicitly involving $V$, one can draw Monte Carlo samples from the Gaussian $\rho$ and approximate them via empirical averages (assuming access to evaluations of the partial derivatives of $V$).

To compute the second term, note that $DT^\lambda$ is diagonal:
\begin{align*}
    DT^\lambda(x)
    = \alpha I_d + \diag{\Bigl( \sum_{j=1}^J \lambda_{i,j} T_j'(x_i) \Bigr)_{i=1}^d}\,.
\end{align*}
Hence, inversion of $DT^\lambda(x)$ is very fast, requiring only $O(Jd)$ time to compute $DT^\lambda(x)$ and then $O(d)$ time to invert it.
Moreover, the $(i,i)$-entry of $(DT^\lambda)^{-1}(x)$ only depends on $x_i$, so the second term reduces to a \emph{one-dimensional integral}:
\begin{align*}
    \int \langle e_i, (DT^\lambda)^{-1}(x)\,e_i\rangle\,T_j'(x_i) \, \rho(\!\dd x)
    &= \int \frac{T_j'(x_i)}{\alpha + \sum_{j'=1}^J \lambda_{i,j'} T_{j'}'(x_i)} \,\rho_1(\!\dd x_i)\,.
\end{align*}
In turn, this one-dimensional integral can be computed rapidly via Gaussian quadrature.

To summarize: the gradient of the potential energy term (the term involving $V$) can be approximated via Monte Carlo sampling, and the gradient of the entropy term decomposes along the coordinates and can therefore be dealt with via standard quadrature rules.
Note that many of these steps can be parallelized.
In~\Cref{sec: sgd_mfvi}, we control the variance of the stochastic gradient, thereby obtaining guarantees for SPGD\@.

To compute the projection of a point $\eta \in \R^{Jd}$ onto the non-negative orthant $\R_+^{Jd}$ w.r.t.\ $\|\cdot\|_Q$, one must solve the following optimization problem:
\begin{align*}
    \min_{\lambda \in \R_+^{Jd}} \ \langle \lambda -\eta, Q\,(\lambda-\eta)\rangle\,.
\end{align*}
Again, due to the block diagonal structure of $Q$, this is equivalent to solving $d$ independent projection problems: in each one, we must project a point in $\R^J$ onto $\R_+^J$ in the $Q^{\cM_1}$-norm.
This is a smooth, convex problem that can itself be solved via, e.g., projected gradient descent, or L-BFGS-B~\citep{zhu1997algorithm}, {or any standard quadratic program solver}.

\subsubsection{Convergence for stochastic mean-field VI}\label{sec: sgd_mfvi}

In~\Cref{sec: implementation_details}, we noted that in general, the gradient of the KL divergence involves an integral over $\rho$, which can be approximated via Monte Carlo sampling.
This leads to a \emph{stochastic} projected gradient algorithm for mean-field VI\@, and this section is devoted to obtaining convergence guarantees for SPGD\@.

Our goal here is not to conduct a comprehensive study, but rather to show how such guarantees can be obtained, and hence we impose a number of simplifying assumptions.
We do not work with the cone augmented by translations, so that the maps are parameterized solely by $\lambda \in \R_+^{|\cM|}$ (the $v$-component is easier to handle and only introduces extra notational burden into the proofs).
Also, we consider a stochastic approximation of the gradient of the potential term via a single sample drawn from $\rho$ at each iteration, and we assume that the gradient of the entropy is computed exactly. As discussed in~\Cref{sec: implementation_details}, the gradient of the entropy can be handled via one-dimensional quadrature.

Even with these simplifications, the variance bound is somewhat involved.
Motivated by the piecewise linear construction of~\Cref{thm: pi_closeness}, in which all maps $T\in \cM$ can be taken to be \emph{bounded}, we impose the following assumption.

\begin{enumerate}[label=$(\Xi)$]
    \item\label{boundedness} There exists $\Xi > 0$ such that for the Gram matrix $Q^{\cM_1}$ associated with the centered \emph{univariate} family $\cM_1$, we have the pointwise bound  $\langle Q^{-1}, \bar{Q}(x) \rangle \leq \Xi J$ for all $x\in\R$, where $\bar{Q}_{T,\tilde{T}}(x) = {T(x) \tilde{T}(x)}$ for $T,\tilde{T} \in \cM_1$.
    Here, $J \defeq |\cM_1|$.
\end{enumerate}

{Similarly to~\cref{lem:Upsilon_lemma}, we can also quantify $\Xi$ for the piecewise linear dictionary.

\begin{lemma}\label{lem:Xi}
    Let $\cM$ be the piecewise linear construction of~\cref{thm: pi_closeness}. Then, $\Xi \lesssim |\cM_1|^2$, where $\cM_1$ is the generating family in a single dimension.
    Since $|\cM_1| = J$, the bound is equivalently $\Xi\lesssim J^2$.
\end{lemma}}

The following lemma established a variance bound of the type~\ref{var_bound} which, when combined with \cref{thm: sgd_theorem}, proves \cref{thm: sgd_vi}.

\begin{lemma}[Variance bound for stochastic mean-field VI]\label{lem: mfvi_sgd_varbound}
Assume that $\pi$ is well-cond\-itioned~\ref{well_cond} and that $\cM$ is generated from a univariate family $\cM_1$ satisfying~\ref{boundedness}. Let $Q^{-1}\,\hat{\nabla}_\lambda\kl{\cdot}{\pi}$ denote the stochastic gradient (see~\Cref{app:mf_var_bd}). Let $\pi^\star_{\diamond}$ denote the unique minimizer of $\kl{\cdot}{\pi}$ over $\cone(\cM;\,\alpha\id)_\sharp \rho$ with $\alpha = 1/\sqrt{L_V}$. Then, the following second moment bound holds:
\begin{align*}
    \E[\tr\Cov(Q^{-1/2}\,\hat{\nabla}_\lambda\kl{\mu_{\lambda}}{\pi})]
    &\leq 2L_V^2\Xi J\, W_2^2(\mu_{\lambda},\pi^\star_\diamond)
    + 4L_V \Xi J \, (L_V \,W_2^2(\pi^\star_\diamond,\pi^\star) + \kappa d)\,.
\end{align*}
\end{lemma}

Let us assume that the $\kappa d$ term is larger than $L_V W_2^2(\pi^\star_\diamond,\pi^\star)$; this can be guaranteed via the approximation result in~\Cref{sec: mfvi_approx}.
The next theorem follows immediately from \cref{thm: sgd_theorem} and the previous lemma.

\begin{theorem}[Convergence of stochastic mean-field VI]\label{thm: sgd_vi}
    Assume that $\pi$ is well-conditioned ~\ref{well_cond} and that $\cM$ is regular~\ref{regular_dict} and generated by a univariate family satisfying~\ref{boundedness}.
    Let $\pi^\star_{\diamond}$ denote the unique minimizer of $\kl{\cdot}{\pi}$ over $\cone(\cM;\,\alpha\id)_\sharp \rho$ with $\alpha = 1/\sqrt{L_V}$. Then, {for all sufficiently small $\eps$, }the iterates of stochastic projected gradient descent yield a measure $\mu_{(t)} $ with the guarantee $\sqrt{\ell_V}\,\E[W_2(\mu_{(t)}, \pi^\star_{\diamond})] \leq \eps$, with a number of iterations bounded by
    \begin{align*}
        t \gtrsim \frac{\Xi \kappa^2 Jd}{\eps^2}\log(\sqrt{\ell_V}\,W_2(\mu_{(0)},\pi^\star_\diamond)/\eps)\,,
    \end{align*}
    and step size $h \asymp \eps^2/(L_V \Xi \kappa Jd)$.
\end{theorem}
{As with Theorem~\ref{thm: acc_vi}, we can state the following corollary given that $\Xi$ is uniformly bounded via \Cref{lem:Xi}.
\begin{corollary}[End-to-end guarantees with SPGD] Consider the setting of Theorem~\ref{thm: sgd_vi}. Then the required runtime to estimate $\pi_\diamond^\star$ becomes 
\begin{align*}
    t = O(dJ^3 \kappa^{2} \eps^{-2}\log(\sqrt{\ell_V}d/\eps))\,.
\end{align*}
If $\pi_\diamond^\star$ is meant to approximation $\pi^\star$, then we can use the approximation guarantees from Theorem~\ref{thm:approx_improved} to obtain the complete convergence guarantee of
\begin{align*}
        t = O({\kappa^{8}}d^{5/2}\eps^{-5}\log(\sqrt{\ell_V}d/\eps))\,.
\end{align*}
\end{corollary}}

\section{Numerical experiments}\label{sec:experiments}
\begin{figure}[h]
\begin{minipage}{0.5\textwidth}
    \centering
    \includegraphics[width=0.75\textwidth]{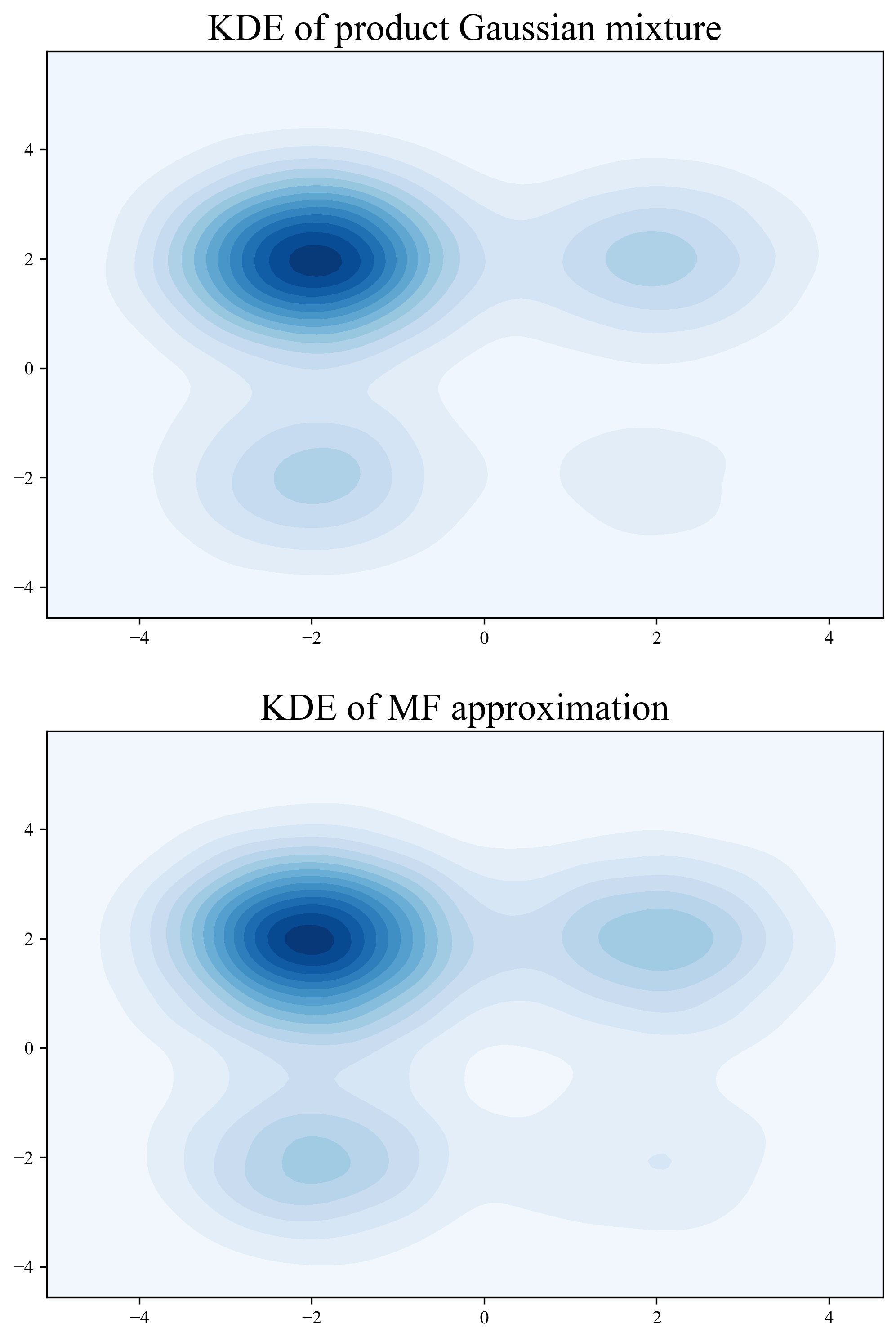}
    \caption{KDEs for the optimal product Gaussian mixture and our algorithm.}\label{fig:2d_gm}
\end{minipage}
\begin{minipage}{0.5\textwidth}
    \centering
    \includegraphics[width=0.75\textwidth]{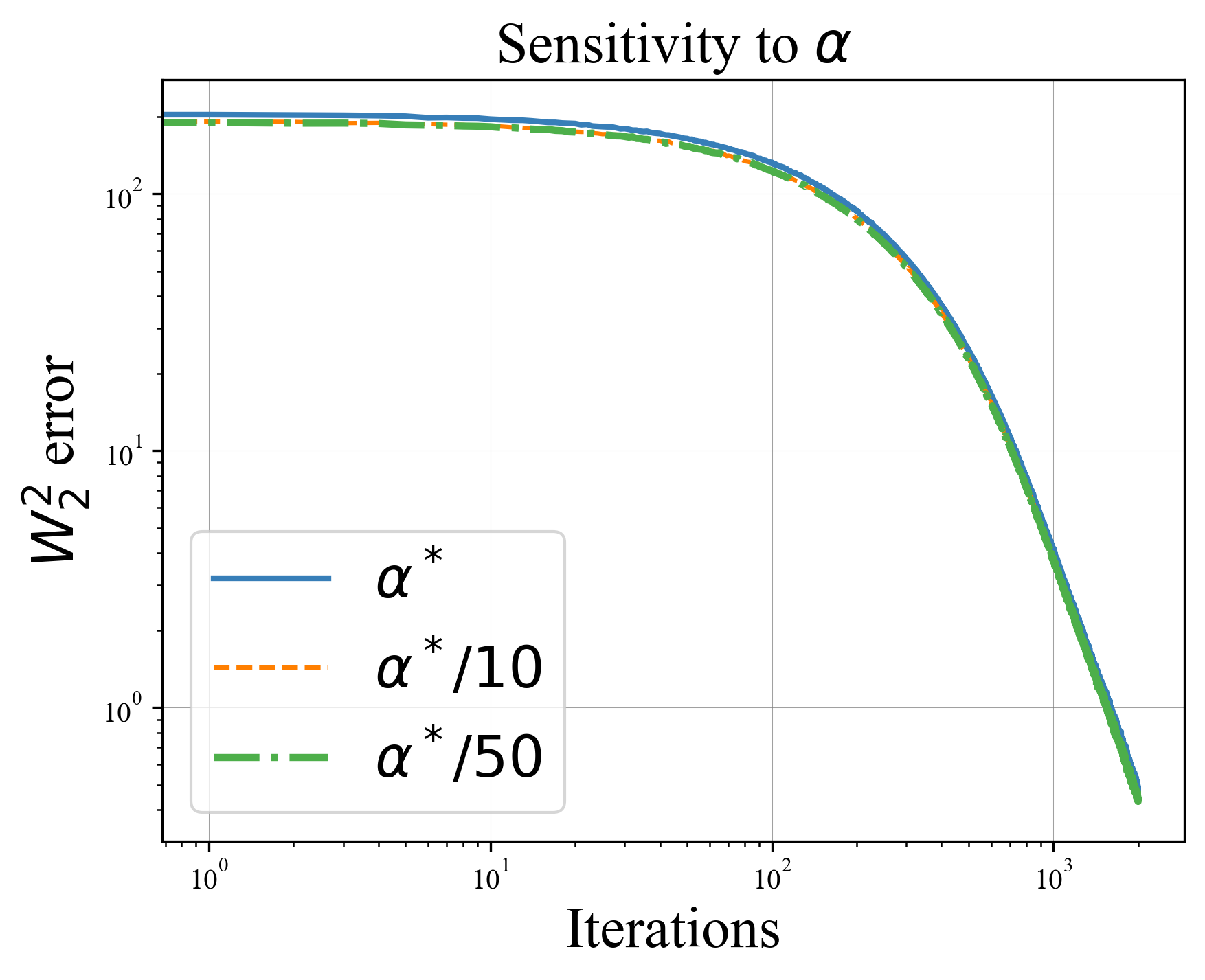}
    \caption{Our algorithm is \textbf{robust to the choice of $\alpha$}.}\label{fig:gaussian}
\end{minipage}
\end{figure}

We showcase our proposed MFVI algorithm on numerical experiments. Experimental details are deferred to~\Cref{sec:experiment_details}, and the code to reproduce the experiments is available \href{https://github.com/APooladian/MFVI}{here}. Across all experiments, which include low- and high-dimensional settings, we took the piecewise linear dictionary (\Cref{thm: pi_closeness}) with the same value for the size $J = |\cM_1| = 28$ of the univariate family (hence $|\cM| = Jd$), and we ran stochastic gradient descent (without acceleration) with a batch size of 2000 samples per iteration.

\subsection{Product Gaussian mixture}
In our first experiment, the target is a mixture of four Gaussians in $\R^2$ which is itself a product measure.
Despite the non-log-concavity, our algorithm correctly recovers the correct target. Though, we note that this approach is sensitive to the initialization, but this is expected as the landscape is non-convex.

\subsection{Non-isotropic Gaussian}
Next, we computed the mean-field approximation of a randomly generated centered and non-isotropic Gaussian in dimension $d=5$. Letting $\Sigma$ denote the covariance matrix, the mean-field approximation is also a Gaussian with diagonal covariance and entries $(\Sigma_{\text{MF}})_{i,i} = 1/(\Sigma^{-1})_{i,i}$ (see \Cref{sec:nonisogaussian_exp} for a calculation of this fact).

In~\Cref{fig:gaussian}, we plot the $W_2^2$ error between the covariance matrix of our algorithm iterate (computed from samples) and $\Sigma_{\rm MF}$, which is a lower bound on the true $W_2^2$ distance~\citep[cf.][]{CueMatTue1996W2Lower}.

In this case, the optimal parameter choice $\alpha^*$ is known, though this is rarely the case in practice. We ran our algorithm for various choices of $\alpha$, fixing all other parameters to be the same. We see that our algorithm does not depend heavily on the choice of hyperparameter $\alpha$, and the practitioner can safely choose a small value of $\alpha$ without sacrificing performance.  

\subsection{Synthetic Bayesian logistic regression}
As a final example, we turn to Bayesian logistic regression on synthetic data; precise details are deferred to \cref{sec:blogreg_exp}. In summary, we are given i.i.d. data $(X_i,Y_i) \in \R^d \times \R$ for $i=1,\ldots,n$, {(where $d=20$ and $n=100$)} from which we want to recover a parameter $\theta$. When assuming an \emph{improper} (Lebesgue) prior on $\theta$, the posterior is given by
\begin{align*}
    V(\theta) = \sum_{i=1}^n \bigl[\log(1 + \exp(\theta^\top X_i)) - Y_i\, \theta^\top X_i\bigr]\,.
\end{align*}
{Note that $V$ is not strongly convex as $V(\theta)$ behaves like a linear function as $\|\theta\| \to +\infty$.}
With $V$ and $\nabla V$ in hand, our algorithm is fully implementable. As we considered an improper prior, a comparison to CAVI is not possible.  Instead, we compared against standard Langevin Monte Carlo (LMC). The final histograms were generated using 2000 samples from both the mean-field VI algorithm and LMC\@. \Cref{fig:BLR_20} contains the 20 marginals for both our approach and LMC, which are closely aligned.

\begin{figure}[h]
    \centering
    \includegraphics[width=0.8\textwidth]{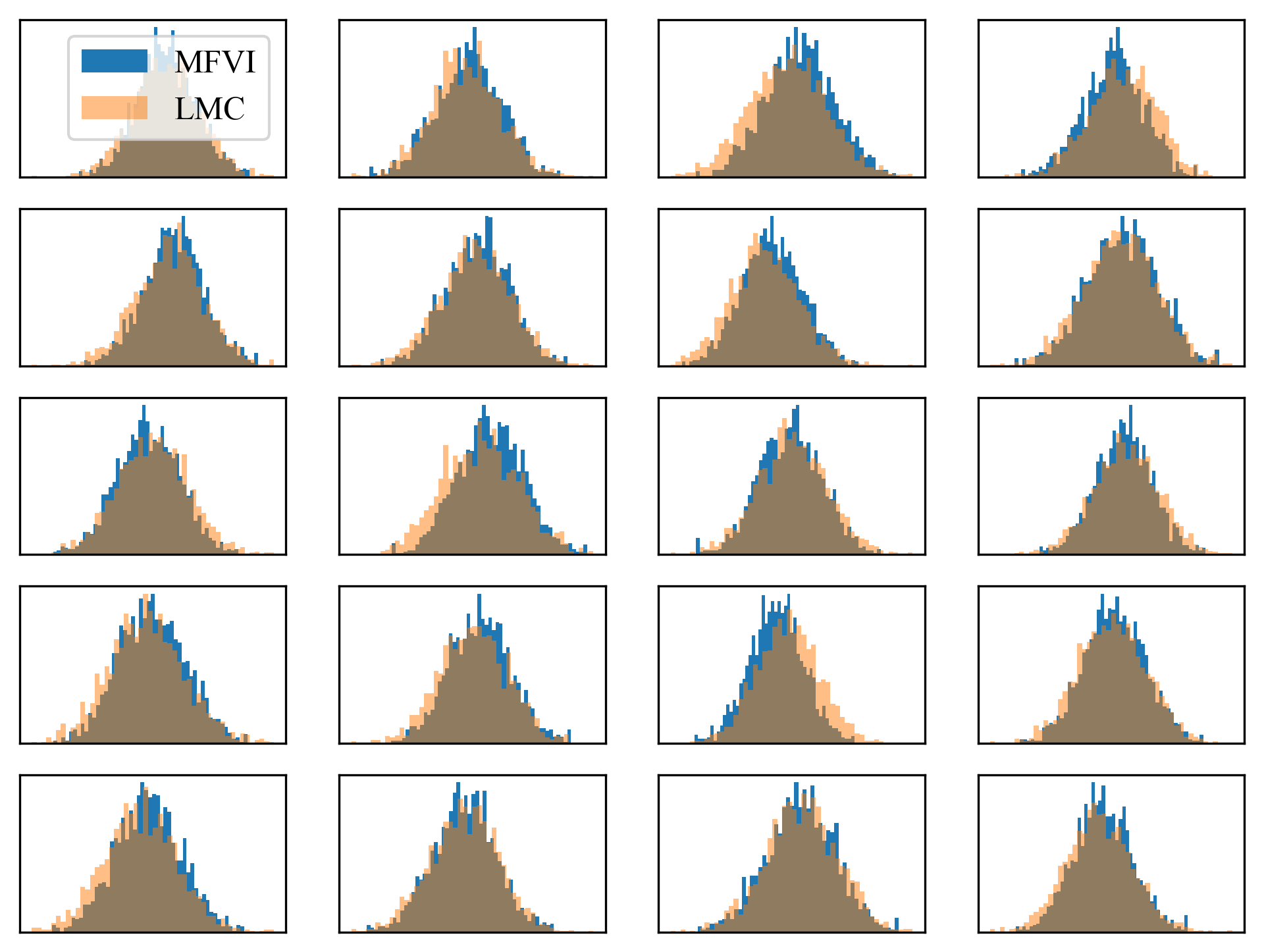}
    \caption{Histograms of the first ten marginals computed via our mean-field VI algorithm vs.\ Langevin Monte Carlo for a $20$-dimensional Bayesian logistic regression example.}\label{fig:BLR_20}
\end{figure}

\section{Extension to mixtures of product measures}\label{sec:mixtures_of_prod}

In this section, we extend our methodology to approximations via mixtures of product measures. The motivation is simply that many more measures can be approximated via mixtures of (approximate) product measures, e.g., Gibbs distributions with small gradient complexity~\citep{Eldan18GradComplexity, EldGro18Decomp, 
Aus19LowComplexity, JaiRisKoe19MeanField}.

In~\Cref{sec: mfvi_convergence}, we minimized $\kl{\cdot}{\pi}$ over $\augtip_\sharp \rho$, where $\augtip$ is parameterized by the pair $(\lambda, v) \in \man \defeq \R_+^{|\cM|} \times \R^d$, equipped with the norm $\norm \cdot_{Q\oplus I_d}$.
In this section, following~\cite{lambert2022variational}, a \emph{mixture of product measures} is specified by a mixing measure $P \in \cP(\man)$ and corresponds to the measure $\mu_P \defeq \int (T^{\lambda,v})_\sharp \rho\, P(\!\dd \lambda,\!\dd v)$.
We can now equip the space $\cP(\man)$ with the Wasserstein geometry (with respect to $\norm \cdot_{Q\oplus I_d}$), and we shall derive the Wasserstein gradient flow of the functional $P \mapsto \kl{\mu_P}{\pi}$.

This approach to mixture modelling is inspired by the distance on Gaussian mixtures proposed in~\citet{CheGeoTan19GMM, DelDes20GMM}; see~\citet{BinBunNil23Mixing} for a statistical perspective.

In this section, we again use the abstract parameterization $T^{\lambda,v} = \alpha \id + \sum_{T\in\cM} \lambda_T T + v$.
Proofs are given in~\Cref{app:pf_mixtures}.

\begin{theorem}\label{thm:wgf_mixtures}
    The Wasserstein gradient flow of $P\mapsto \kl{\mu_P}{\pi}$ is the flow ${(P^{(t)})}_{t\ge 0}$ specified as follows.
    For each $t\ge 0$, $P^{(t)}$ is the law of $(\lambda^{(t)}, v^{(t)})$, where
    \begin{align*}
        \dot \lambda^{(t)}_T
        &= -\int \bigl\langle \nabla \log \frac{\mu_{P^{(t)}}}{\pi} \circ T^{\lambda^{(t)}, v^{(t)}}, T\bigr\rangle\dd \rho\,, &&\text{for}~T\in\cM\,, \\
        \dot v^{(t)}
        &= -\int \nabla \log \frac{\mu_{P^{(t)}}}{\pi}\circ T^{\lambda^{(t)}, v^{(t)}}\dd\rho\,.
    \end{align*}
\end{theorem}

In practice, we use a finite number $K$ of mixture components, in which case
\begin{align}\label{eq:mixture_equal_weights}
    P = \frac{1}{K}\sum_{k=1}^K \delta_{(\lambda[k], v[k])}\,, \qquad \mu_P = \frac{1}{K} \sum_{k=1}^K (T^{\lambda[k], v[k]})_\sharp \rho\,.
\end{align}
The system of ODEs above then becomes an interacting particle system:
\begin{align*}
    \dot\lambda_T^{(t)}[k]
    &= -\int \bigl\langle \nabla \log \frac{\mu_{P^{(t)}}}{\pi} \circ T^{\lambda^{(t)}[k], v^{(t)}[k]}, T\bigr\rangle\dd \rho\,, &&\text{for}~T\in\cM\,, \\
    \dot v^{(t)}[k]
    &= -\int \nabla \log \frac{\mu_{P^{(t)}}}{\pi}\circ T^{\lambda^{(t)}[k], v^{(t)}[k]}\dd\rho\,.
\end{align*}
The particles interact through the common term $\log \mu_{P^{(t)}}$.
More explicitly, by the change of variables formula,
\begin{align*}
    \mu_P = \frac{1}{K} \sum_{k=1}^K \frac{\rho \circ (T^{\lambda[k],v[k]})^{-1}}{\det D T^{\lambda[k],v[k]} \circ (T^{\lambda[k],v[k]})^{-1}}\,.
\end{align*}
{Note that computing $\nabla \log \mu_P$ now requires taking the second derivative of the transport maps, which hinders implementation. In this case, a smooth family $\cM$ is required.}

The dynamics~\eqref{eq:mixture_equal_weights} maintains equal weights for each of the particles at each time.
We can similarly derive the gradient flow with respect to the Wasserstein{--}Fisher{--}Rao {(or Hellinger--Kantorovich)} geometry, which captures unbalanced optimal transport~\citep{LieMieSav16HK, Chietal18FR, LieMieSav18HK}.
The use of this geometry for sampling was pioneered in~\citet{LuLuNol19BirthDeath}.

\begin{theorem}\label{thm:wfr_flow}
    The Wasserstein{--}Fisher{--}Rao gradient flow of $P\mapsto \kl{\mu_P}{\pi}$, initialized at $P^{(0)} = \sum_{k=1}^K w^{(0)}[k]\,\delta_{(\lambda^{(0)}[k], v^{(0)}[k])}$ with $\sum_{k=1}^K w^{(0)}[k] = 1$, can be described as follows.
    For each time $t\ge 0$, let $P^{(t)} = \sum_{k=1}^K w^{(t)}[k] \,\delta_{(\lambda^{(t)}[k], v^{(t)}[k])}$ and $r^{(t)}[k] \defeq \sqrt{w^{(t)}[k]}$.
    Then,
    \begin{align*}
        \dot\lambda_T^{(t)}[k]
        &= -\int \bigl\langle \nabla \log \frac{\mu_{P^{(t)}}}{\pi} \circ T^{\lambda^{(t)}[k], v^{(t)}[k]}, T\bigr\rangle\dd \rho\,, &&\text{for}~T\in\cM\,, \\
        \dot v^{(t)}[k]
        &= -\int \nabla \log \frac{\mu_{P^{(t)}}}{\pi}\circ T^{\lambda^{(t)}[k], v^{(t)}[k]}\dd\rho\,, \\
        \dot r^{(t)}[k]
        &= -\Bigl(\int \log \frac{\mu_{P^{(t)}}}{\pi}\circ T^{\lambda^{(t)}[k], v^{(t)}[k]}\dd \rho - \int \log\frac{\mu_{P^{(t)}}}{\pi}\dd \mu_{P^{(t)}}\Bigr)\,r^{(t)}[k]\,.
    \end{align*}
\end{theorem}

We leave it as an open question to obtain convergence rates for this flow.

\section{Conclusion}

In this paper, we have put forth a definition of polyhedral subsets of the Wasserstein space, resting upon the notion of compatibility.
Compatibility induces an isometry with Euclidean geometry, allowing for the application of scalable first-order algorithms for convex optimization.
We then applied our framework to yield end-to-end guarantees for mean-field VI\@, showing first that the mean-field solution lies arbitrarily close to the KL minimizer over a polyhedral set, and then providing complexity guarantees for computing that KL minimizer.

Looking ahead, our work leaves open several questions which could be fruitful for future study.
Regarding our notion of Wasserstein polyhedra, one could
investigate statistical convergence guarantees as in \cite{gunsilius2022tangential}, or develop further interesting applications of polyhedral optimization.

As for our application to mean-field VI\@, one direction to explore would be to see if our analysis of SGD can be improved, e.g., by removing the boundedness assumption in~\ref{boundedness}.
Finally, another important question is to move beyond the well-conditioned setting.

\subsection*{Acknowledgements}
The authors are grateful to Jonathan Niles-Weed for helpful discussions at all stages of this work, as well as three reviewers that greatly improved the quality of our manuscript. YJ acknowledges support from NYU through the SURE program, where this project originated under the supervision of AAP. SC acknowledges the support of the Eric and Wendy Schmidt Fund at the Institute for Advanced Study. AAP acknowledges the support of Meta AI research, as well as the National Science Foundation through NSF Award 1922658.

\bibliography{main}

\pagebreak
\appendix

\section{Proofs for Section~\ref{sec: wass_polytopes}}\label{app: main_proofs}

\begin{proof}[Proof of \cref{lem: mut_diag}]
Take $T_1(x) = A_1x$ and $T_2(x) = A_2x $ for $A_1,A_2$ positive definite, and mutually diagonalizable: there exists an orthogonal matrix $U$ such that $A_i = U \Lambda_i U^{-1}$ with $\Lambda_i$ diagonal with positive entries. Then
\begin{align*}
    (T_1 \circ (T_2)^{-1})(x) = U \Lambda_1 U^{-1}\,(U \Lambda_2 U^{-1} )^{-1}\,x = U \Lambda_1 \Lambda_2^{-1} U^{-1} x = \tilde{A}x\,,
\end{align*}
with $\tilde{A} \succ 0$; this completes the claim.
\end{proof}

\begin{proof}[Proof of \cref{lem: radial}]
    See~\citet[Section 2.3.2]{panaretos2020invitation}.
\end{proof}

\begin{proof}[Proof of \cref{lem: 1d_maps}]
Let $S, T \in \cM$, and for simplicity assume they are strictly increasing. Note that $T^{-1}$ is also strictly increasing, so $S \circ T^{-1}$ is strictly increasing.
\end{proof}

\begin{proof}[Proof of \cref{lem: direct_sum}]
Take $S_1, T_1 \in \cM_1$ and $S_2, T_2 \in \cM_2$. Take $(x,y) \in \R^{d_1\times d_2}$, and write $S(x,y) = (S_1(x),S_2(y))$, and similarly for $T$. Since each of $S_1 \circ T_1^{-1}$ and $S_2 \circ T_2^{-1}$ are gradients of convex functions, then $S \circ T^{-1} = (S_1 \circ T_1^{-1}, S_2 \circ T_2^{-1})$ is also the gradient of a convex (and separable) function.
\end{proof}

\begin{proof}[Proof of \cref{lem:add_id}]
For any $T \in \cM$, $T$ and $T^{-1}$ are both gradients of convex functions, so the claim is immediate.    
\end{proof}
\begin{proof}[Proof of \cref{lem:translations}]
Suppose $T_1,T_2 \in \cM$ are compatible i.e., $T_1 \circ (T_2)^{-1}$ is the gradient of a convex function. Write $\tilde{T}_1 = \nabla\tilde{\phi}_1 = \nabla(\phi_1 + \langle u, \cdot\rangle)$ and $\tilde{T}_2 = \nabla\tilde{\phi}_2 = \nabla(\phi_2 + \langle v, \cdot\rangle)$. One can check that $\tilde{\phi}_2^*(y) = \phi_2^*(y-v)$, and then by convex duality $(\tilde{T}_2)^{-1} = \nabla \phi_2^*(\cdot-v) $ is the gradient of a convex function. So,
\begin{align*}
    \tilde{T}_1(\tilde{T}^{-1}_2(y)) = \nabla\phi_1(\nabla\phi_2^*(y-v)) + u\,,
\end{align*}
which is the gradient of a sum of convex functions.
\end{proof}

\begin{proof}[Proof of \cref{lem: cones}]
For $\eta, \lambda \in \R^{|\cM|}_+$, write $S^\eta = \sum_{S \in \cM}\eta_S S$ and $T^\lambda = \sum_{T \in \cM}\lambda_T T$ in $\coneM$. Assume $\eta, \lambda \ne 0$ or otherwise the statement is trivial. The composition reads
\begin{align*}
    T^\lambda \circ (S^\eta)^{-1} = \textstyle\sum_{T \in \cM}\lambda_T T \circ \left(\textstyle\sum_{S \in \cM}\eta_S S\right)^{-1}\,,
\end{align*}
so it suffices to show that $\tilde{T} \defeq T \circ \left(\textstyle\sum_{S \in \cM}\eta_S S\right)^{-1}$ is the gradient of a convex function. Since each $S \in \cM$ is the gradient of a convex function, we have that 
\begin{align*}
    \tilde{T}^{-1} = \bigl(\sum_{S \in \cM} \eta_S S\bigr) \circ T^{-1} = \sum_{S \in \cM} \eta_S\, (S \circ T^{-1})\,.
\end{align*}
Since $\tilde{T}^{-1}$ is the gradient of a convex function, by conjugacy, it holds that $\tilde{T}$ is the gradient of a convex function.
\end{proof}

\section{Proofs for Section~\ref{sec:opt_algs}}\label{app: sgd_general}

\begin{proof}[Proof of \cref{thm: sgd_theorem}]
For an iteration number $t \in \NN$, we use the shorthand $\hat\nabla_\lambda \cF_t \defeq \hat\nabla_\lambda {\cF}(\mu_{\lambda^{(t)}})$, and similarly for the true gradient. 

Since projections are contractive, a first manipulation gives
\begin{align*}
    \|\lambda^{(t+1)} - \lambda^\star\|^2_Q \leq \|\lambda^{(t)} - \lambda^\star\|^2_Q + h^2\, \|Q^{-1}\,\hat\nabla_\lambda {\cF}_t\|_Q^2 + 2h\, \langle \hat\nabla_\lambda {\cF}_t, \lambda^\star - \lambda^{(t)} \rangle\,.
\end{align*}
Taking expectations conditioned on $\lambda^{(t)}$ yields, by linearity,
\begin{align*}
    \E_t\|\lambda^{(t+1)} - \lambda^\star\|^2_Q
    &\leq \|\lambda^{(t)} - \lambda^\star\|^2_Q + h^2\, \E_t\|Q^{-1}\,\hat\nabla_\lambda {\cF}_t\|_Q^2 + 2h\, \langle \nabla_\lambda {\cF}_t, \lambda^\star - \lambda^{(t)} \rangle\,.
\end{align*}
By $m$-strong convexity of $\cF$, we obtain
\begin{align*}
    \E_t\|\lambda^{(t+1)} - \lambda^\star\|^2_Q
    &\leq (1 - mh)\,\|\lambda^{(t)} - \lambda^\star\|^2_Q + h^2\, \E_t\|Q^{-1}\,\hat\nabla_\lambda {\cF}_t\|_Q^2 + 2h\,(\cF(\mu_\star) - \cF(\mu_{\lambda^{(t)}})) \\
    &\leq  (1 - 2mh)\|\lambda^{(t)} - \lambda^\star\|^2_Q + h^2\, \E_t\|Q^{-1}\,\hat\nabla_\lambda {\cF}_t\|_Q^2\,.
\end{align*}
Taking expectations again, 
\begin{align*}
    \E\|\lambda^{(t+1)} - \lambda^\star\|^2_Q
    &\leq  (1 - 2mh)\,\E\|\lambda^{(t)} - \lambda^\star\|^2_Q + h^2\, \E[\E_t\|Q^{-1}\,\hat\nabla_\lambda {\cF}_t\|_Q^2]\,.
\end{align*}
Adding and subtracting the true gradient at iterate $\lambda^{(t)}$, written $Q^{-1}\,\nabla_\lambda \mathcal{F}_t$, the second term can be bounded via smoothness of $\mathcal{F}$:
\begin{align*}
    h^2\,\E[\E_t\|Q^{-1}\,\hat\nabla_\lambda {\cF}_t\|_Q^2]
    &\leq 2h^2\,\E[\E_t\|Q^{-1}\,(\hat\nabla_\lambda {\cF}_t - \nabla_\lambda {\cF}_t)\|_Q^2] + 2h^2\,\E\|Q^{-1}\,\nabla_\lambda {\cF}_t\|^2_Q \\
    &\leq 2h^2\,\E[\E_t\|Q^{-1}\,(\hat\nabla_\lambda {\cF}_t - \nabla_\lambda {\cF}_t )\|_Q^2] + 2M^2 h^2\,\E\|\lambda^{(t)} - \lambda^\star\|^2_Q\,. 
\end{align*}
Combining this with our previous bound results in
\begin{align*}
    \E\|\lambda^{(t+1)} - \lambda^\star\|^2_Q
    &\leq  (1 - 2mh + 2M^2 h^2)\,\E\|\lambda^{(t)} - \lambda^\star\|^2_Q + 2h^2\, \E[\E_t\|Q^{-1}\,(\hat\nabla_\lambda {\cF}_t - \nabla_\lambda {\cF}_t )\|_Q^2] \\
    &\leq (1-mh)\,\E\|\lambda^{(t)} - \lambda^\star\|^2_Q + 2h^2\, \E[\E_t\|Q^{-1}\,(\hat\nabla_\lambda {\cF}_t - \nabla_\lambda {\cF}_t )\|_Q^2]\,,
\end{align*}
where in the last step we took $h \leq \frac{1}{2 \kappa M}$.

By \ref{var_bound}, we obtain
\begin{align*}
    \E\|\lambda^{(t+1)} - \lambda^\star\|^2_Q
    &\leq  (1 - mh + c_1 h^2)\,\E\|\lambda^{(t)} - \lambda^\star\|^2_Q + c_0 h^2\,.
\end{align*}
If $c_1h^2 \leq mh/2$, i.e., $h \leq m/(2c_1)$, then 
\begin{align*}
    \E\|\lambda^{(t+1)} - \lambda^\star\|^2_Q
    &\leq  (1 - mh/2)\,\E\|\lambda^{(t)} - \lambda^\star\|^2_Q +  c_0 h^2\,.
\end{align*}
Iterating this bound gives
\begin{align*}
    \E\|\lambda^{({t})} - \lambda^\star\|^2_Q
    &\leq  (1 - \tfrac{mh}{2})^t\,\|\lambda^{(0)} - \lambda^\star\|^2_Q + \frac{2c_0 h}{m} \\
    &\leq \exp(-mht/2)\,\|\lambda^{(0)} - \lambda^\star\|^2_Q + \frac{2c_0 h}{m}\,.
\end{align*}
Choosing $h \asymp m\eps^2/c_0$ concludes the proof.
\end{proof}

\section{Proofs for Section~\ref{sec: mfvi_main}}\label{app: vi_proofs}

\subsection{Proofs for Section~\ref{sec: mfvi}}\label{app:mf_eq_pf}

To derive the mean-field equations, we recall that the KL divergence is
\begin{align*}
    \kl{\mu}{\pi}
    &= \int V\dd\mu + \int \log \mu\dd \mu + \log Z\,.
\end{align*}
Over the space of product measures, we obtain the functional
\begin{align*}
    (\mu_1,\dotsc,\mu_d)
    &\mapsto \KL\Bigl(\bigotimes_{i=1}^d \mu_i \Bigm\Vert \pi\Bigr)
    = \int V\dd\bigotimes_{i=1}^d \mu_i + \sum_{i=1}^d \int \log\mu_i\dd\mu_i + \log Z\,.
\end{align*}
If we take the first variation of this functional~\citep[c.f.][Section 7.2]{San15} w.r.t.\ $\mu_i$, we obtain the equation
\begin{align*}
    \Bigl[\delta_{\mu_i}\KL\Bigl(\bigotimes_{j=1}^d \mu_j \Bigm\Vert \pi\Bigr)\Bigr](x_i)
    &= \int V(x_1,\dotsc,x_d) \, \bigotimes_{j\ne i} \mu_j(\!\dd x_j) + \log\mu_i(x_i) + \text{const}\,.
\end{align*}
At optimality, the first variation must equal a constant, which leads to
\begin{align*}
    \pi_i^\star(x_i)
    &\propto\exp\Bigl(-\int V(x_1,\dotsc,x_d) \,\bigotimes_{j\ne i} \pi_j^\star(\!\dd x_j)\Bigr)\,.
\end{align*}

\subsection{Proofs for Section~\ref{sec:regularity}}\label{app:regularity_proofs}

In this section, we prove the regularity bounds on the optimal transport maps given as~\Cref{thm:regularity}.
Recall that $\pi^\star$ denotes the mean-field VI solution and $T^\star$ is the optimal transport map from $\rho$ to $\pi^\star$.
Let $\pi^\star_i$ and $T^\star_i$ denote the $i$-th components respectively, and recall also from~\eqref{eq: mf_equations} that $\pi^\star_i \propto\exp(-V_i)$, where
\begin{align*}
    V_i(x_i)
    &\defeq \int V(x_1,\dotsc,x_d)\,\bigotimes_{j\ne i} \pi^\star_j(\!\dd x_j)\,.
\end{align*}
We begin with a few simple lemmas which show that $T_i^\star(0)$, the mean of $\pi_i^\star$, and the mode of $\pi_i^\star$ are all close to each other.

\begin{lemma}\label{lem:T_near_mean}
    Let $T$ denote the optimal transport map from $\rho = \cN(0,1)$ to $\mu$, and let $m$ denote the mean of $\mu$. If $T' \leq \beta$, then $|T(0) - m| \leq \sqrt{2/\uppi}\,\beta$.
\end{lemma}
\begin{proof}
    Let $Z \sim \cN(0,1)$, so that $T(Z) \sim \mu$ and $m \defeq \E T(Z)$. Since $T' \leq \beta$, 
\begin{align*}
    |T(0) - m|
    &= |\E(T(0) - T(Z))| \leq \beta\, \E|Z| = \sqrt{\frac{2}{\uppi}}\,\beta\,.\qedhere
\end{align*}
\end{proof}

\begin{lemma}\label{lem:mean_near_mode}
    Let $m$ and $\tilde m$ denote the mean and the mode of $\mu$, respectively, where $\mu$ is $\ell_V$-strongly log concave and univariate.
    Then, $|m-\tilde m| \le 1/\sqrt{\ell_V}$.
\end{lemma}
\begin{proof}
    This is a standard consequence of strong log-concavity, see, e.g.,~\citet[Proposition 4]{DalKarRio22NonStrongly}.
\end{proof}

We are now ready to prove~\Cref{thm:regularity}.

\begin{proof}[Proof of~\Cref{thm:regularity}]
As the main text contains the proof of the bounds on the first derivative of $T$, we continue with the second and third derivative bounds.

We, obviously, start with the second derivative bounds. Recall the Monge--Amp{\`e}re equation (or the change of variables formula) yields
\begin{align}\label{eq: 1d_mongeampere}
    \log \pi_i^\star \circ T_i^\star (x) = -\frac{x^2}{2} - \log (T_i^\star)'(x) - \tfrac12\log(2\uppi)\,.
\end{align}
Differentiating once yields
\begin{align}\label{eq:monge_ampere_deriv}
(\log \pi^\star_i \circ T^\star_i)'(x)  = {-}V_i'(T^\star_i(x))\, (T_i^\star)'(x) = - x - \frac{(T_i^\star)''(x)}{(T_i^\star)'(x)}\,.
\end{align}
Rearranging to isolate $(T_i^\star)''$ gives
\begin{align}\label{eq:reg_pf_1}
    (T_i^\star)''(x) = -(T_i^\star)'(x)\,\bigl(x {-} V_i'(T_i^\star(x))\, (T^\star_i)'(x)\bigr)\,.
\end{align}
Let $m_i^\star$ and $\tilde m_i^\star$ denote the mean and mode of $\pi_i^\star$ respectively.
Recall also that $0 < 1/\sqrt{L_V} \leq (T_i^\star)' \leq 1/\sqrt{\ell_V}$.
By~\Cref{lem:T_near_mean} and~\Cref{lem:mean_near_mode},
\begin{align*}
    |V_i'(T_i^\star(x))|
    &\le {\underbrace{|V_i'(\tilde m_i^\star)|}_{=0}} + L_V\,|T_i^\star(x) - \tilde m_i^\star| \\
    &\le L_V\,(|T_i^\star(x) - T_i^\star(0)| + |T_i^\star(0) - m_i^\star| + |m_i^\star - \tilde m_i^\star|) \\
    &\le L_V \, \bigl(\frac{1}{\sqrt{\ell_V}}\,|x| + \sqrt{\frac{2}{\uppi}}\,\frac{1}{\sqrt{\ell_V}} + \frac{1}{\sqrt{\ell_V}} \bigr)
    \lesssim \frac{L_V}{\sqrt{\ell_V}}\,(1+|x|)\,.
\end{align*}
Substituting this into~\eqref{eq:reg_pf_1}, we obtain
\begin{align*}
    |(T_i^\star)''(x)|
    &\lesssim \frac{1}{\sqrt{\ell_V}}\,\bigl(|x| + \frac{L_V}{\ell_V}\,(1+|x|)\bigr)
    \lesssim \frac{\kappa}{\sqrt{\ell_V}}\,(1+|x|)\,.
\end{align*}

For the third derivative control, we differentiate~\eqref{eq:monge_ampere_deriv} again to yield
\begin{align*}
    (\log \pi_i^\star \circ T_i^\star)''(x)
    &= {-} \bigl(V_i''(T^\star_i(x))\,(T_i^\star)'(x)^2 + V_i'(T^\star_i(x))\,(T_i^\star)''(x)\bigr)\\
    &= -1 - \frac{(T_i^\star)'''(x)\, (T_i^\star)'(x) - (T_i^\star)''(x)^2}{(T_i^\star)'(x)^2}\,.
\end{align*}
Again, we rearrange and isolate, giving
\begin{align*}
    (T_i^\star)'''(x)
    &= \frac{(T_i^\star)''(x)^2}{(T_i^\star)'(x)} - (T_i^\star)'(x)\, \bigl(1 { - } V_i''(T_i^\star(x))\, (T_i^\star)'(x)^2 { - } V_i'(T_i^\star(x))\,(T_i^\star)''(x)\bigr)\,.
\end{align*}
Taking absolute values, we can collect the terms one by one:
\begin{align*}
    |(T_i^\star)''(x)^2/(T_i^\star)'(x)|
    &\lesssim \frac{\kappa^{{2}}}{\sqrt{\ell_V}}\,(1+|x|^2)\,, \\
    |V_i''(T^\star_i(x))\,(T_i^\star)'(x)^2|
    &\leq \kappa\,,\\
    |V_i'(T_i^\star(x))\,(T_i^\star)''(x)| 
    &\lesssim \frac{L_V}{\sqrt{\ell_V}}\,(1+|x|)\cdot \frac{\kappa}{\sqrt{\ell_V}}\,(1 + |x|) \lesssim \kappa^2\,(1 + |x|^2)\,.
\end{align*}
{To obtain the first bound, note that by~\eqref{eq:reg_pf_1} and the subsequent calculations, we have $|(T_i^\star)''(x)/(T_i^\star)'(x)| \lesssim \kappa\,(1+|x|)$. Square and use $(T_i^\star)'(x) \le 1/\sqrt{\ell_V}$.}
Hence, the final bound scales as
\begin{align*}
    |(T^\star_i)'''(x)|
    &\lesssim \frac{\kappa^2}{\sqrt{\ell_V}}\,(1 + |x|^2)\,. \qedhere
\end{align*}
\end{proof}

\subsection{Proofs for Section~\ref{sec: mfvi_approx}}\label{app:approx_proofs}

For our approximation results, we begin with a simple construction via piecewise linear maps.
Let $R > 0$ denote a truncation parameter, and partition the interval $[-R, +R]$ into sub-intervals of length $\delta > 0$.
Let $\psi$ be the elementary step function
\begin{align*}
    \psi : \R\to\R\,, \qquad \psi(x) \defeq \begin{cases}
        0\,, & x \le 0\,, \\
        x\,, & x \in [0,1]\,, \\
        1\,, & x \ge 1\,.
    \end{cases}
\end{align*}
We then define the following family of compatible maps:
\begin{align*}
    \cM
    &\defeq \bigl\{ x \mapsto \psi(\delta^{-1}\,(x_i - a))\, e_i \bigm\vert i \in [d], \; I = [a,a+\delta]~\text{is a sub-interval} \bigr\}\,.
\end{align*}
We suppress the dependence on the parameters $R$, $\delta$ in the notation.

\begin{proof}[Proof of~\Cref{thm: pi_closeness}]
    Owing to the isometry, we wish to show that we can find a map $\hat T \in \augtip$, with $\alpha = 1/\sqrt{L_V}$, such that
    \begin{align}\label{eq:desired_approx}
        \|\bar T - \hat T\|_{L^2(\rho)}^2 \le {\varepsilon^2/\ell_V}\qquad\text{and}\qquad\|D(\bar T - \hat T)\|_{L^2(\rho)}^2 \le {\varepsilon_1^2/\ell_V}\,.
    \end{align}
    Here, $\|D(\bar T - \hat T)\|_{L^2(\rho)}^2 \defeq \int \|D(\bar T - \hat T)\|_{\rm F}^2 \dd \rho$.
    
    We first make a series of reductions.
    By assumption, $D\bar T \succeq \alpha I$, and by definition, $\hat T$ is of the form $\alpha \,{\id} + \sum_{T\in\cM} \lambda_T T + v$.
    By replacing $\bar T$ with $\bar T - \alpha\id$, it suffices to prove the following statement: assuming that $0 \preceq D\bar T \preceq \ell_V^{-1/2}\,I$ together with the second derivative bound on $\bar T$, there exists $\hat T$ of the form $\sum_{T\in\cM} \lambda_T T + v$ such that~\eqref{eq:desired_approx} holds.
    However, from the structure of $\cM$, the problem now separates across the coordinates and it suffices to prove this statement with $d=1$ {and $\varepsilon$ replaced with $\varepsilon/\sqrt d$}.

    \textbf{Truncation procedure.} We will construct $\hat T$ so that $\bar T(-R) = \hat T(-R)$ and $\bar T(+R) = \hat T(+R)$.
    Assuming that this holds, the bound on $\bar T'$ and the fact that $\hat T$ is constant on $(-\infty, -R]$ and on $[+R, +\infty)$ readily imply
    \begin{align*}
        |\bar T(x) - \hat T(x)| \le \frac{1}{\sqrt{\ell_V}}\,(|x|-R)\,, \qquad\text{for}~|x|\ge R\,.
    \end{align*}
    The error contributed by the tails is therefore bounded by 
    \begin{align*}
        \int_{\R \setminus (-R, +R)} |\bar T - \hat T|^2 \dd \rho
        \le \frac{1}{\sqrt{2\uppi}\,\ell_V} \int_{\R\setminus (-R, +R)} {(|x|-R)}^2 \exp(-x^2/2) \dd x\,.
    \end{align*}
    Similarly, 
    \begin{align*}
        |\bar T'(x) - \hat T'(x)| \le 1/\sqrt{\ell_V}\,, \qquad\text{for}~|x|\ge R\,,
    \end{align*}
    which gives
    \begin{align*}
        \int_{\R \setminus (-R, +R)} |\bar T' - \hat T'|^2 \dd \rho
        &\le \frac{1}{\sqrt{2\uppi}\,\ell_V} \int_{\R\setminus (-R, +R)} \exp(-x^2/2) \dd x\,.
    \end{align*}
    
    Standard Gaussian tail bounds and the Cauchy{--}Schwarz inequality imply that with the choice $R\asymp \sqrt{\log(1/(\ell_V \varepsilon^2))}$, we obtain {$\|\bar T - \hat T\|_{L^2(\tilde\rho)}^2 
 \vee \|\bar T' - \hat T'\|_{L^2(\tilde\rho)}^2 \lesssim \varepsilon^2$, where $\tilde\rho$ is the Gaussian measure restricted to the set $\R \setminus [-R,R]$.}

    \textbf{Uniform approximation over a compact domain.}
    We now show that $\hat T$ can be chosen to uniformly approximate $\bar T$ on $[-R,+R]$.
    Indeed, we take
    \begin{align*}
        \hat T(x)
        &= \bar T(-R) + \sum_{m=0}^{2R/\delta-1} \lambda_m \psi\Bigl( \frac{x - (-R + m\delta)}{\delta}\Bigr)\,,
    \end{align*}
    where the $\lambda_m$ are chosen so that $\bar T$ and $\hat T$ agree at each of the endpoints of the sub-intervals of size $\delta$.
    Consider such a sub-interval $I = [a,a+\delta]$.
    Then, for $x\in I$,
    \begin{align*}
        |\bar T(x) - \hat T(x)|
        &= \Bigl\lvert \bar T(x) - \bar T(a) - \frac{\bar T(a+\delta) - \bar T(a)}{\delta}\,(x-a) \Bigr\rvert\,.
    \end{align*}
    By the mean value theorem, $\bar T(x) = \bar T(a) + \bar T'(c_1)\,(x-a)$ and $\bar T(a+\delta) = \bar T(a) + \bar T'(c_2) \, \delta$ for some $c_1, c_2 \in I$. Together with the second derivative bound on $\bar T$, it yields 
    \begin{align*}
        |\bar T(x) - \hat T(x)|
        &= |(\bar T'(c_1) - \bar T'(c_2))\,(x-a)|
        \lesssim \frac{\kappa R}{\sqrt{\ell_V}}\,\delta^2\,.
    \end{align*}
    Similarly, for the derivative,
    \begin{align*}
        |\bar T'(x) - \hat T'(x)|
        &= \Bigl\lvert \bar T'(x) - \frac{\bar T(a + \delta) - \bar T(a)}{\delta}\Bigr\rvert
        = |\bar T'(x) - \bar T'(c_2)|
        \lesssim \frac{\kappa R}{\sqrt{\ell_V}} \,\delta\,.
    \end{align*}
    To obtain our desired error bounds, we take $\delta = \widetilde\Theta(\sqrt{\varepsilon/\kappa})$.
    Finally, to obtain the stated bounds in the theorem in dimension $d$, replace $\varepsilon$ with $\varepsilon/\sqrt d$.

    {With this choice of $\delta$, we then obtain $\varepsilon_1 = \widetilde O(\sqrt d\,\kappa \delta) = \widetilde O(\kappa^{1/2} d^{1/4} \varepsilon^{1/2})$.}

    \textbf{Size of the generating family.} Finally, the size of $\cM$ is $O({d} R/\delta) = \widetilde O(\kappa^{1/2} d^{{5}/4}/\varepsilon^{1/2})$, which completes the proof.
\end{proof}

In the proof above, we have used the bounds on the first and second derivatives of $\bar T$.
However, from~\Cref{thm:regularity}, we actually have control on the third derivative as well, so we can expect to exploit this added degree of smoothness to obtain better approximation rates.

As above, we fix a truncation parameter $R > 0$ and a mesh size $\delta > 0$.
Our family of maps will be constructed from the following basic building blocks.
\begin{itemize}
    \item \textbf{Linear function.} We let $\psi^{\rm lin}(x) \defeq x$ for $x\in\R$.
    \item \textbf{Piecewise quadratics.} Define the piecewise quadratic
    \begin{align*}
        \psi^{\rm quad,\pm}(x) \defeq \pm \begin{cases}
            0\,, & x \le 0\,, \\
            x^2\,, & x \in [0,1]\,, \\
            2x-1\,, & x \ge 1\,.
        \end{cases}
    \end{align*}
    \item \textbf{Piecewise cubics.} Define the piecewise cubic,
    \begin{align*}
        \psi^{\rm cub,\pm}(x) \defeq \pm\begin{cases}
            0\,, & x \le 0\,, \\
            x^2\,(3-2x)\,, & x \in [0,1]\,, \\
            1\,, & x \ge 1\,.
        \end{cases}
    \end{align*}
\end{itemize}
Given a univariate function $\psi$ and $i\in [d]$, we extend it to a map $\psi_i : \R^d\to\R^d$ by setting $\psi_i(x) \defeq \psi(x_i)$.
Also, given a sub-interval $I = [a,a+\delta]$, we define the map $\psi_{I,i} : \R^d\to\R^d$ via $\psi_{I,i}(x) \defeq \psi(\delta^{-1}\,(x_i-a))$.

Let $\cI$ denote the set of sub-intervals.
Our generating family will consist of
\begin{align*}
    \cM
    &\defeq \{\psi^{\rm lin}_i \mid i \in [d]\} \cup \bigcup_{I\in \cI} \bigcup_{i=1}^d {\{\psi^{\rm quad,-}_{I,i}, \psi^{\rm quad,+}_{I,i}, \psi^{\rm cub,-}_{I,i}, \psi^{\rm cub,+}_{I,i}\}}\,,
\end{align*}
which consists of $(4\,|\cI| + 1)\,d$ elements.
However, we will not consider the full cone generated by $\cM$---indeed, if we did, then the presence of the \emph{negative} piecewise quadratics and cubics would mean that we obtain non-monotone maps.

Elements of our polyhedral set will be of the form $x\mapsto \alpha \id + \sum_{T\in\cM}\lambda_T T + v$, where $v\in\R^d$ and we may decorate components of $\lambda$ according to the elements of $\cM$ to which they correspond, e.g., $\lambda_{I,i}^{\rm quad,-}$ is the coefficient in front of $\psi_{I,i}^{\rm quad,-}$.

To provide some intuition, we will use the linear function and the piecewise quadratics to approximate the \emph{derivative} of $\bar T$.
Indeed, suppose for the moment that $\bar T$ is univariate and note that the derivatives of the linear and piecewise quadratic functions give rise to piecewise linear interpolations of $\bar T'$.
The interpolation of $\bar T'$, once integrated, does not necessarily interpolate $\bar T$, and the piecewise cubics will be used to remedy this issue.

Toward this end, note that since $\bar T$ is monotonically increasing, $\bar T'$ is non-negative.
We will want our approximating $\hat T$ to have the same property, which will be ensured by imposing \emph{linear constraints} on $\lambda$.
We consider the following polyhedral subset of $\R_+^{|\cM|}$:
\begin{align}\label{eq:higher_approx_polyhedral}
    \begin{aligned}
        K
        \defeq \Bigl\{\lambda \in \R_+^{|\cM|} &\Bigm\vert \forall i\in [d],\; \frac{2}{\delta}\sum_{I\in\cI} (\lambda_{I,i}^{\rm quad,+}-\lambda_{I,i}^{\rm quad,-}) + \lambda^{\rm lin}_i \ge 0\,, \\
        &\qquad\qquad \text{and}\qquad\forall I \in \cI,\; \forall i \in [d], \; \frac{6\lambda^{\rm cub,-}}{\delta} \le \frac{\alpha}{2}\Bigr\}\,.
    \end{aligned}
\end{align}
We then take $\cK \defeq \{x \mapsto \alpha x + \sum_{T\in T} \lambda_T T + v \mid \lambda \in K, \; v\in\R^d\}$ and $\Pdiam \defeq \cK_\sharp \rho$.
The first constraint ensures that the sum of the linear and piecewise quadratic functions has non-negative slope.
As for the second constraint, it ensures that the sum of the negative piecewise cubic functions has slope at least $-\alpha/2$.
Since we always add $\alpha \id$, each of our maps will have slope at least $\alpha/2$ and therefore be increasing.
With this choice, our family consists of gradients of strongly convex functions with convexity parameter \emph{less} than that of the true map $T^\star$, which does affect some of the other results of this paper (e.g., the geodesic smoothness of the KL divergence in~\Cref{prop: kl_smooth}), but only by at most a constant factor, and henceforth we ignore this technical issue.

We are now ready to prove our improved approximation result.

\begin{proof}[Proof of~\Cref{thm:higher_order}]
    We start with the same reductions as in the proof of~\Cref{thm: pi_closeness}, reducing to the univariate case.

    \textbf{Truncation procedure.}
    The truncation procedure is similar to the one before, except that $\hat T$ is no longer constant on $(-\infty, -R]$ and on $[+R,+\infty)$.
    Instead, on these intervals, $\hat T$ will be linear, with the additional conditions $\bar T'(-R) = \hat T'(-R)$ and $\bar T'(+R) = \hat T'(+R)$.
    However, the arguments still go through, and we can take $R\asymp\sqrt{\log(1/(\ell_V \varepsilon^2))}$ as before.

    \textbf{Uniform approximation over a compact domain.}
    We will first construct a preliminary version of $\hat T$ without using the piecewise cubics.
    Recall from the discussion above that using the linear and piecewise quadratic functions, we can ensure that $\hat T'$ is a linear interpolation of $\bar T'$.
    Namely, we set
    \begin{align*}
        \hat T' = \bar T'(-R) + \sum_{I \in \cI} \bigl[\lambda_I^{\rm quad,-} \,(\psi^{\rm quad,-})' + \lambda_I^{\rm quad,+}\,(\psi^{\rm quad,+})'\bigr]\,,
    \end{align*}
    where the coefficients are chosen such that $\bar T'$ and $\hat T'$ agree at each of the endpoints of the sub-intervals.
    Following the argument as before, 
    for a sub-interval $I = [a,a+\delta]$ and $x\in I$,
    \begin{align*}
        |\bar T'(x) - \hat T'(x)|
        &= \Bigl\lvert \bar T'(x) - \bar T'(a) - \frac{\bar T'(a+\delta) - \bar T'(a)}{\delta}\,(x-a)\Bigr\rvert\,.
    \end{align*}
    By the mean value theorem, $\bar T'(x) = \bar T'(a) + \bar T''(c_1)\,(x-a)$ and $\bar T'(a+\delta) = \bar T'(a) + \bar T''(c_2)\,\delta$ for some $c_1, c_2 \in I$.
    Using the bounds on the derivatives of $\bar T$,
    \begin{align}\label{eq:higher_approx_error}
        |\bar T'(x) - \hat T'(x)|
        &= |(\bar T''(c_1) - \bar T''(c_2))\,(x-a)|
        \lesssim \frac{\kappa^2 R^2}{\sqrt{\ell_V}}\,\delta^2\,.
    \end{align}

    Next, we wish to control $|\bar T(x) - \hat T(x)|$. Here, $\hat T$ is defined by integrating $\hat T'$, and choosing the shift $v$ so that $\bar T(-R) = \hat T(-R)$.
    First, \emph{suppose} that $\bar T(a) = \hat T(a)$.
    We can then use the fundamental theorem of calculus to obtain
    \begin{align}\label{eq:endpoint_error}
        |\bar T(x) - \hat T(x)|
        &= \Bigl\lvert \int_a^x (\bar T'(y) - \hat T'(y))\dd y\Bigr\rvert
        \lesssim \frac{\kappa^2 R^2}{\sqrt{\ell_V}}\,\delta^3\,.
    \end{align}
    In particular, $|\bar T(a+\delta) - \hat T(a+\delta)|$ is of order $\delta^3$.

    To ensure that $\bar T$ and $\hat T$ agree at each of these endpoints, we scan the set of sub-intervals left to right, and we iteratively add non-negative multiples of the piecewise cubics in order to achieve this interpolating condition.
    Since the original endpoint error is bounded in~\eqref{eq:endpoint_error}, it follows that the coefficients of the piecewise cubics that we add are small: $0 \le \lambda^{\rm cub,\pm}_I \lesssim \kappa^2 R^2 \delta^3/\sqrt{\ell_V}$.
    In particular, the constraint on $\lambda^{\rm cub,-}_I$ in~\eqref{eq:higher_approx_polyhedral} is met for small $\delta$.

    The key property of the piecewise cubics is that $(\psi^{\rm cub,\pm})'(0) = (\psi^{\rm cub,\pm})'(1) = 0$.
    This means that even after adding the piecewise cubics, $\bar T'$ and $\hat T'$ agree at all of the endpoints of the sub-intervals.
    However, we must check that adding these piecewise cubics does not destroy the approximation rates~\eqref{eq:higher_approx_error} and~\eqref{eq:endpoint_error}.
    Since $|(\psi^{\rm cub,\pm}_I)'| \lesssim 1/\delta$, the bound on the coefficients for the piecewise cubics shows that the derivative of the piecewise cubic part of $\hat T$ is bounded in magnitude by $O(\kappa^2 R^2 \delta^2/\sqrt{\ell_V})$, so that~\eqref{eq:higher_approx_error} is intact.
    Similarly,~\eqref{eq:endpoint_error} is also intact, either by integrating~\eqref{eq:higher_approx_error} or by using the bound on the coefficients of the piecewise cubics.
    {Thus, $|\bar T(x) - \hat T(x)| \lesssim \kappa^2 R^2 \delta^3/\sqrt{\ell_V}$, and setting this to be at most $\varepsilon/\sqrt{d\ell_V}$ yields the choice $\delta = \widetilde\Theta(\varepsilon^{1/3}/(\kappa^{2/3} d^{1/6}))$.}

    {\textbf{Size of the generating family.} The size of the generating family is then $O(dR/\delta) = \widetilde O(\kappa^{2/3} d^{7/6}/\varepsilon^{1/3})$, which completes the proof.}
\end{proof}

Using the bounds on the Jacobian $D\hat T_\diamond$ of the approximating map, we can bound the change in the KL divergence on the path from $\hat\pi_\diamond$ to $\pi^\star$.
This shows that $\hat\pi_\diamond$ has a small \emph{suboptimality gap} for KL minimization over $\Pdiam$.
The following calculation is similar to the one for~\Cref{prop: kl_smooth}, which establishes smoothness of the KL divergence over $\Pdiam$. However, since $\pi^\star$ does not lie in $\Pdiam$, it does not apply here.

\begin{corollary}\label{cor: kl_diamond_approx}
    Assume that $\pi$ is well-conditioned~\ref{well_cond}.
    Let $\hat\pi_\diamond = (\hat T_\diamond)_\sharp \rho$ denote the approximation to $\pi^\star$ given by the piecewise linear construction (\Cref{thm: pi_closeness}).
    Then,
    \begin{align*}
        \kl{\hat\pi_\diamond}{\pi} - \kl{\pi_\diamond^\star}{\pi}
        &\le \kl{\hat\pi_\diamond}{\pi} - \kl{\pi^\star}{\pi}
        \lesssim \kappa^3 d^{1/2} \varepsilon\,.
    \end{align*}
    If, on the other hand, $\hat\pi_\diamond = (\hat T_\diamond)_\sharp \rho$ is given by the construction of~\Cref{thm:higher_order},
    \begin{align*}
        \kl{\hat\pi_\diamond}{\pi} - \kl{\pi^\star_\diamond}{\pi}
        &\le \kl{\hat\pi_\diamond}{\pi} - \kl{\pi^\star}{\pi}
        \lesssim \kappa^{10/3} d^{1/3} \varepsilon^{4/3}\,.
    \end{align*}
\end{corollary}
\begin{proof}
    Let $(\mu_t)_{t\in [0,1]}$ denote the geodesic joining $\pi^\star$ to $\hat\pi_\diamond$.
    Then, by differentiating the KL divergence along this geodesic twice, we obtain the following expressions; see \citet{ChewiBook} and \citet[Appendix B.2]{diao2023forward} for derivations.
    We write $T = \hat T_\diamond \circ (T^\star)^{-1}$ for the optimal transport map from $\pi^\star$ to $\hat\pi_\diamond$, and $T_t = (1-t)\,{\id} + t\,T$.
    
    For the potential energy term,
    \begin{align*}
        \partial_t^2 \cV(\mu_t)
        &= \E_{\pi^\star}\langle T-{\id}, (\nabla^2 V \circ T_t)\,(T-{\id})\rangle
        \le L_V\,\E_{\pi^\star}\|T-{\id}\|^2
        = L_V\,\E_{\rho}\|\hat T_\diamond - T^\star\|^2\,.
    \end{align*}
    Next, for the entropy term,
    \begin{align*}
        \partial_t^2 \cH(\mu_t)
        &= \E_{\pi^\star}\|(DT_t)^{-1}\,(DT - I)\|_{\rm F}^2\,,
    \end{align*}
    By~\Cref{thm:regularity}, $DT = D((T^\star)^{-1})\,D\hat T_\diamond \succeq 1/\sqrt\kappa$, so $DT_t \succeq 1/\sqrt\kappa$. Also, $DT^\star \succeq 1/\sqrt{L_V}$. Therefore, we obtain
    \begin{align*}
        \partial_t^2 \cH(\mu_t)
        &\le \kappa \,\E_{\pi^\star}\|D((T^\star)^{-1})\,D\hat T_\diamond \circ (T^\star)^{-1} - I\|_{\rm F}^2 \\
        &\le \kappa L_V\,\E_{\pi^\star}\|D\hat T_\diamond \circ (T^\star)^{-1} - D((T^\star)^{-1})\|_{\rm F}^2
        = \kappa L_V\,\E_{\rho}\|D\hat T_\diamond - DT^\star\|_{\rm F}^2\,.
    \end{align*}
    Therefore, adding the two terms together,
    \begin{align*}
        \partial_t^2 \kl{\mu_t}{\pi}
        &\le L_V\,\|\hat T_\diamond - T^\star\|_{L^2(\rho)}^2 + \kappa L_V\,\|D(\hat T_\diamond - T^\star)\|_{L^2(\rho)}^2\,.
    \end{align*}

    Integrating this expression from $t=0$ to $t=1$,
    \begin{align*}
        \kl{\hat\pi_\diamond}{\pi} - \kl{\pi^\star}{\pi}
        &\le \E_{\pi^\star}\langle [\nabla_{\WW}\kl{\cdot}{\pi}](\pi^\star), T-{\id}\rangle \\
        &\qquad{} + \frac{L_V}{2}\,(\|\hat T_\diamond - T^\star\|_{L^2(\rho)}^2 + \kappa\,\|D\hat T_\diamond - DT^\star\|_{L^2(\rho)}^2)\,.
    \end{align*}
    However, since $\hat\pi_\diamond$, $\pi^\star$ both belong to the geodesically convex set of product measures, and $\pi^\star$ minimizes the KL divergence over this set, we must have $\E_{\pi^\star}\langle [\nabla_{\WW}\kl{\cdot}{\pi}](\pi^\star), T-{\id}\rangle = 0$.

    We are now in a position to apply the approximation guarantees.
    Applying the result of~\Cref{thm: pi_closeness}, we obtain
    \begin{align*}
        \kl{\hat\pi_\diamond}{\pi} - \kl{\pi^\star}{\pi}
        &\lesssim \kappa \varepsilon^2 + \kappa^3 d^{1/2}\varepsilon\,.
    \end{align*}
    If we instead use the improved guarantee of~\Cref{thm:higher_order}, we obtain
    \begin{align*}
        \kl{\hat\pi_\diamond}{\pi} - \kl{\pi^\star}{\pi}
        &\lesssim \kappa \varepsilon^2 + \kappa^{10/3} d^{1/3} \varepsilon^{4/3}\,. \qedhere
    \end{align*}
\end{proof}

Finally, from the small suboptimality gap of $\hat\pi_\diamond$ and the strong geodesic convexity of the KL divergence, we are able to prove that $\pi^\star$ is close, not just to our constructed $\hat\pi_\diamond$, but to the minimizer $\pi^\star_\diamond$ of the KL divergence over $\Pdiam$, which in turn can be computed via the algorithms in~\Cref{sec: mfvi_convergence}.

\begin{proof}[Proof of~\Cref{thm:approx_improved}]
By triangle inequality, we have
\begin{align*}
    W_2(\pi^\star_\diamond,\pi^\star) \leq W_2(\pi^\star_\diamond,\hat\pi_\diamond)  + W_2(\hat\pi_\diamond, \pi^\star)\,, 
\end{align*}
and since we can control the second term (recall \cref{thm: pi_closeness}), it suffices to control the first. Since $\kl{\cdot}{\pi}$ is $\ell_V$-strongly geodesically convex, the first term can be bounded above by
\begin{align*}
   \ell_V W_2^2(\pi^\star_\diamond,\hat\pi_\diamond)/2 \leq \kl{\hat\pi_\diamond}{\pi} - \kl{\pi^\star_\diamond}{\pi} \lesssim \kappa^3d^{1/2}\tilde\eps\,,
\end{align*}
where the final bound is obtained from \cref{cor: kl_diamond_approx} (we only take the worst-case scaling term), and $\tilde\varepsilon$ is the approximation accuracy guaranteed by~\Cref{thm: pi_closeness}. Setting this equal to $\eps^2$, we apply~\Cref{thm: pi_closeness} with $\frac{\eps^2}{\kappa^3 d^{1/2}}$ replacing $\varepsilon$ and we see that $|\cM| = \widetilde{O}(\kappa^2 d^{3/2}/\varepsilon)$.

    Similarly, for the higher-order approximation scheme, we use~\Cref{cor: kl_diamond_approx} and apply~\Cref{thm:higher_order} with $\frac{\varepsilon^{3/2}}{\kappa^{5/2} d^{1/4}}$ replacing $\varepsilon$, obtaining $|\cM| = \widetilde O(\kappa^{3/2} d^{5/4}/\varepsilon^{1/2})$.
\end{proof}

\subsection{Proofs for Section~\ref{sec: mfvi_convergence}}

\begin{proof}[Proof of \cref{prop: kl_smooth}]
We write 
\begin{align*}
    \kl{\mu}{\pi} = \cV(\mu) + \cH(\mu) \defeq \int V \dd \mu + \int \log \mu\dd\mu +  \log Z\,.
\end{align*}
To prove smoothness, it suffices to show that the Wasserstein Hessians for both $\cV$ and $\cH$ are bounded.
Since we work with the augmented cone, we let
\begin{align*}
    T^{\lambda, v}
    &\defeq \alpha \id + \sum_{T\in\cM} \lambda_T T + v\,, \qquad \mu_{\lambda,v} \defeq {(T^{\lambda,v})}_\sharp \rho\,.
\end{align*}
Our goal is to upper bound the following quadratic forms
\begin{align*}
    & \nabla_{\WW}^2 \cV(\mu_{\lambda,v})[T_{\lambda,v}^{\eta,u} - \text{id},T_{\lambda,v}^{\eta,u} - \text{id}] = \E_{\mu_{\lambda,v}} [(T_{\lambda,v}^{\eta,u} - {\id})^\top\, \nabla^2 V\,(T_{\lambda,v}^{\eta,u} - {\id})]\,, \\
    &\nabla_{\WW}^2 \cH(\mu_{\lambda,v})[T_{\lambda,v}^{\eta,u} - \text{id},T_{\lambda,v}^{\eta,u} - \text{id}] = \E_{\mu_{\lambda,v}}\|DT_{\lambda,v}^{\eta,u} - I\|^2_{\rm F}\,,
\end{align*}
in terms of the squared Wasserstein distance between $\mu_{\lambda,v}$ and $\mu_{\eta,u}$, and $T_{\lambda,v}^{\eta,u}$ is the optimal transport map from $\mu_{\lambda,v}$ to $\mu_{\eta,u}$. See \citet{ChewiBook} and \citet[Appendix B.2]{diao2023forward} for derivations of these expressions. We bound the two terms separately. 

An upper bound on the potential term is straightforward. By \ref{well_cond}, $\nabla^2V \preceq L_V I$, and so
\begin{align*}
    \nabla_{\WW}^2 \cV(\mu_{\lambda,v})[T_{\lambda,v}^{\eta,u} - \text{id},T_{\lambda,v}^{\eta,u} - \text{id}]
    &= \E_{\mu_{\lambda,v}} [(T_{\lambda,v}^{\eta,u} - {\id})^\top\, \nabla^2 V\,(T_{\lambda,v}^{\eta,u} - {\id})] \\
    &\leq L_V\, \E_{\mu_{\lambda,v}}\|T_{\lambda,v}^{\eta,u} - {\id}\|^2
    = L_V\,W_2^2(\mu_{\lambda,v}, \mu_{\eta,u})\,.
\end{align*}

The entropy term needs a bit more work. To start, 
we note that by compatibility,
\begin{align}\label{eq: comp_app}
    T_{\lambda,v}^{\eta,u}
    = T^{\eta,u} \circ (T^{\lambda,v})^{-1}
    = T^\eta \circ (T^\lambda)^{-1}(\cdot - v) + u\,,
\end{align}
where we write $T^{\lambda,v} = T^\lambda + v$ and similarly $T^{\eta,u} = T^\eta + u$.
By the chain rule,
\begin{align*}
    DT_{\lambda,v}^{\eta,u}(\cdot)
    &= [DT^\eta \circ (T^\lambda)^{-1}(\cdot - v)] \, D[(T^\lambda)^{-1}](\cdot - v)\,.
\end{align*}
(For simplicity, the reader may wish to first read the following calculations setting $u=v=0$.)
Performing the appropriate change of variables, the Wasserstein Hessian of $\cH$ reads
\begin{align*}
    \E_{\mu_{\lambda,v}}\|DT_{\lambda,v}^{\eta,u} - I\|^2_{\rm F}
    &= \E_{\mu_{\lambda,v}}\|[DT^\eta \circ (T^\lambda)^{-1}(\cdot - v)] \, D[(T^\lambda)^{-1}](\cdot - v) - I\|^2_{\rm F}\\
    &= \E_{\rho}\| DT^\eta\,D[(T^\lambda)^{-1}] \circ (T^{\lambda,v} - v) - I\|_{\rm F}^2 \\
    &= \E_{\rho}\|DT^\eta\, D[(T^\lambda)^{-1}] \circ T^\lambda  - I\|^2_{\rm F} \\
    &= \E_{\rho}\|DT^\eta\, (D T^\lambda)^{-1}  - I\|^2_{\rm F}\,,
\end{align*}
where we invoked the inverse function theorem in the last step. Given our set of maps, we know that for any $\lambda \in \R^{|\cM|}_+$, $DT^\lambda \succeq \alpha I$, and since $DT^\lambda \, (DT^\lambda)^{-1} = I$, we obtain
\begin{align*}
    \E_{\mu_{\lambda,v}}\|DT_{\lambda,v}^{\eta,u} - I\|^2_{\rm F}
    \leq \frac{1}{\alpha^2}\,\E_{\rho}\|DT^\eta - DT^\lambda\|^2_{\rm F}\,.
\end{align*}
Since our maps are regular (i.e., \ref{regular_dict} holds), there exists $\Upsilon > 0$ such that
\begin{align*}
    \E_\rho\|DT^\eta - DT^\lambda\|_{\rm F}^2
    = \E_\rho\Bigl\lVert \sum_{T\in\cM} (\lambda_T - \eta_T)\,DT\Bigr\rVert_{\rm F}^2
    &= \langle \eta - \lambda, Q^{(1)}\,(\eta-\lambda)\rangle \\
    &\le \Upsilon\, \langle \eta - \lambda, Q\,(\eta-\lambda)\rangle\,.
\end{align*}
Finally, note that
\begin{align*}
    W_2^2(\mu_{\lambda,v}, \mu_{\eta,u})
    = \E_\rho\Bigl\lVert \sum_{T\in\cM} (\eta_T - \lambda_T)\,T + u - v\Bigr\rVert^2
    &= \E_\rho\Bigl\lVert \sum_{T\in\cM} (\eta_T - \lambda_T)\,T \Bigr\rVert^2 + \|u-v\|^2 \\
    &= \langle \eta - \lambda, Q\,(\eta-\lambda)\rangle + \|u-v\|^2\,,
\end{align*}
where we used the fact that the maps in $\cM$ are \emph{centered}.
This shows that
\begin{align*}
    \nabla_{\WW}^2 \cH(\mu_{\lambda,v})[T_{\lambda,v}^{\eta,u} - \text{id},T_{\lambda,v}^{\eta,u} - \text{id}]
    &\le \frac{\Upsilon}{\alpha^2}\,W_2^2(\mu_{\lambda,v}, \mu_{\eta,u})\,.
\end{align*}
Combining all of the terms completes the proof.
\end{proof}

{
\begin{proof}[Proof of Lemma~\ref{lem:Upsilon_lemma}]
We restrict our attention to the piecewise linear family denoted $\cM$ in dimension one with $|\cM| = J$. This suffices due to the tensorization property of $\Upsilon$, see the remark after the definition of $\Upsilon$. It suffices to prove, for all $\lambda \in \R^J$, 
    \begin{align*}
        \bigl\|\sum_{T \in \cM} \lambda_T T' \bigr\|_{L^2(\rho)}^2 \leq \Upsilon\, \bigl\|\sum_{T \in \cM} \lambda_T T \bigr\|_{L^2(\rho)}^2\,,
    \end{align*}
    where $\rho = \cN(0,1)$. We truncate the domain of $\rho$ to $[-R,R]$, where $R \asymp \sqrt{\log(1/(\ell_V \eps^2))}$. On some interval $[a, a+\delta]$, note that
    \begin{align*}
        T^\lambda(x) = T^\lambda(a) + \lambda_T\,((x-a)/\delta)_+\,, \quad DT^\lambda(x) = \lambda_T/\delta\,.
    \end{align*}
    It thus suffices to prove the statement on such an interval. This is equivalent to proving that
    \begin{align*}
        \int_a^{a+\delta} \Bigl(\frac{\lambda_T}{\delta}\Bigr)^2\, \rho({\rm d} x) \leq \Upsilon \int_a^{a+\delta} \Bigl(T^\lambda(a) + \lambda_T\, \frac{x-a}{\delta}\Bigr)^2\,\rho({\rm d}x)\,.
    \end{align*}
    Rearranging, it suffices to show that
    \begin{align*}
        \delta^{-2}\Upsilon^{-1} \leq \frac{\inf_{m\in\mathbb R} \int_a^{a+\delta}  \bigl(\frac{x-a}{\delta} - m\bigr)^2\,\rho({\rm d}x) }{\int_a^{a+\delta} \rho({\rm d}x)} = \frac{{ \rm{var}}X}{\delta^2}\,,
    \end{align*}
    or $\Upsilon^{-1} \leq {\rm{var}}X$, with $X \sim \rho|_{[a,a+\delta]}$.

    Letting $m_a \defeq \E X$, suppose WLOG $m_a  \leq a + \delta/2$. We compute
    \begin{align*}
        \E[(X - m_a)^2] \geq \E[(\delta/4)^2\, \bm{1}_{X \geq a + 3\delta/4}] \gtrsim \delta^2\,\mathbb{P}(X \geq a + 3\delta/4) = \delta^{2}\,\frac{\int_{a + 3\delta/4}^{a+\delta}\rho({\rm d}x)}{\int_{a}^{a+\delta}\rho({\rm d}x)}
        \gtrsim \delta^2\,,
    \end{align*}
    provided $\delta \lesssim 1/R$; indeed, for this choice of $\delta$, $|{\log \rho(x)-\log\rho(y)}| \lesssim 1$ for all $x,y \in [a,a+\delta]$. Stringing together the inequalities, we obtain the desired claim.
\end{proof}
}

\subsection{Proofs for Section~\ref{sec: sgd_mfvi}}\label{app:mf_var_bd}

In this section, we prove our variance bounds for SPGD for mean-field VI\@. We start with a gradient bound under $\pi^\star$.

\begin{lemma}\label{lem: mfvi_nablav_lemma}
    Let $\pi$ be a \ref{well_cond} measure, and let $\pi^\star$ be the mean-field approximation. Then 
    \begin{align}
        \E_{\pi^\star}\nabla V = 0\,, \qquad \E_{\pi^\star}\|\nabla V\|^2 \le L_V\kappa d\,.
    \end{align}
\end{lemma}
\begin{proof}
{Recall our definition of $\pi^\star$ with components $\pi_i \propto \exp(-V_i)$ with
\begin{align*}
    V_i(x_i) \int_{\R^{d-1}}V(x_1,\ldots,x_d)\, \bigotimes_{j\neq i} \pi^\star_j(\!\dd x_j)
\end{align*}
}
    Assuming the first claim, we can prove the second by applying the Brascamp{--}Lieb inequality~\citep{BraLie1976}:
    \begin{align*}
        \E_{\pi^\star}\|\nabla V - \E_{\pi^\star}\nabla V \|^2 \leq \E_{\pi^\star}\tr((\nabla^2 V)^2\, \text{diag}(\vec{V}'')^{-1} )\,, 
    \end{align*}
    where $\vec{V}''\defeq (V_1'',\ldots,V_d'')$. By~\Cref{prop: mfe_properties}, each component satisfies the bound $(V_i'')^{-1} \leq 1/\ell_V$, and we also have by assumption $\nabla^2V\preceq L_V I$. Together, the bound is clear:
    \begin{align*}
        \E_{\pi^\star}\|\nabla V \|^2 \leq \tr((L_VI)^2)/\ell_V = L_V\kappa d\,.
    \end{align*}
    It remains to prove the first equality. Recall that for $i \in [d]$, 
    \begin{align*}
        V_i(x_i) = \int V(x)\, \bigotimes_{j \neq i}\pi^\star_j(\!\dd x_j)\,.
    \end{align*}
    Consider a test vector $e_1 = (1,0,\ldots,0) \in \R^d$. Appropriately interchanging the order of integration, one can check that we obtain
    \begin{align*}
        \E_{\pi^\star}\nabla V^\top e_1 = \int V_1'(x_1)\, \pi^\star_1(\!\dd x_1) =  \int V_1'(x)\, \frac{\exp(-V_1(x_1))}{\int\exp(-V_1(x_1'))\dd x_1'}\dd x_1 = 0\,,
    \end{align*}
    by an application of integration by parts. The same is true for the other coordinates.
\end{proof}

\begin{proof}[Proof of \cref{lem: mfvi_sgd_varbound}]
We want to bound the quantity
\begin{align*}
    \E[\|Q^{-1}\,(\hat{\nabla}_\lambda \cV(\mu_\lambda) - {\nabla}_\lambda \cV(\mu_\lambda))\|_Q^2] = \E[\|Q^{-1/2}\,(\hat{\nabla}_\lambda \cV(\mu_\lambda) - {\nabla}_\lambda \cV(\mu_\lambda))\|^2]\,.
\end{align*}
Using convenient notation choices, we first recall the expressions of the stochastic and non-stochastic gradients of the potential energy:
\begin{align*}
    \hat{\nabla}_\lambda \cV(\mu_\lambda) = \bm{T}(\hat{X})\,\nabla V(T^\lambda(\hat X)) \,, \qquad \nabla_\lambda \cV(\mu_\lambda) = \E_\rho[\bm{T} \,\nabla V\circ T^\lambda]\,,
\end{align*}
where $\hat{X} \sim \rho$ is a random draw, and $\bm{T}(\hat{X}) = (T_1(\hat{X}),\ldots,T_{|\cM|}(\hat{X})) \in \R^{|\cM|} \times \R^d$ is the evaluation of the \emph{whole} dictionary at the random draw. 

We begin by exploiting symmetry in the problem, reducing it to one dimension. First, note that $\bm{T}$ can be equivalently expressed as $d$ repetitions of the following vectors,
\begin{align*}
    \bm{T} = (T_{1:J},\ldots,T_{1:J})\,,
\end{align*}
where $T_{1:J}$ denotes the first $J$ maps in our dictionary (the same maps exist in all dimensions)
(This is a slight abuse of notation because the $i$-th occurrence of $T_{1:J}$ above acts only on the $i$-th coordinate of the input.)
Thus, the matrix $Q^{-1/2}$ is block-diagonal, written
\begin{align*}
    Q^{-1/2} = I_d \otimes Q_{1:J}^{-1/2}\,,
\end{align*}
where $Q_{1:J}$ is the first $J \times J$ block of the full $Q$ matrix, which is $Jd \times Jd$. We can similarly express the gradients with respect to $\lambda$ in this way (i.e., only differentiating the first $J$ components), which results in controlling the following quantity
\begin{align*}
    \E[\|Q^{-1/2}\,(\hat{\nabla}_\lambda \cV(\mu_{\lambda}) - {\nabla}_\lambda \cV(\mu_{\lambda}))\|^2] = \sum_{i=1}^d \E[\|Q_{1:J}^{-1/2}\,(\hat{\nabla}_{1:J} \cV(\mu_{\lambda}) - {\nabla}_{1:J} \cV(\mu_{\lambda}))\|^2]\,.
\end{align*}
Combining these reductions, we are left with bounding the following term in each dimension:
\begin{align*}
    \tr \Cov\bigl(Q_{1:J}^{-1/2} \, T_{1:J}(\hat X_i) \, \partial_i V(T^\lambda(\hat X)\bigr)
    &= \E\bigl[\langle T_{1:J}(\hat X_i)\,T_{1:J}(\hat X_i)^\top, Q_{1:J}^{-1}\rangle \,\partial_i V(T^\lambda(\hat X))^2\bigr] \\
    &\le \Xi J\, \E[\partial_i V(T^\lambda(\hat X))^2]\,,
\end{align*}
where we invoked \ref{boundedness} in the last inequality.
Summing over the coordinates,
\begin{align*}
    \E[\|Q^{-1/2}\,(\hat{\nabla}_\lambda \cV(\mu_{\lambda}) - {\nabla}_\lambda \cV(\mu_{\lambda}))\|^2]
    &\le \Xi J\, \E_\rho\|\nabla V\circ T^\lambda\|^2\,.
\end{align*}

We bound the remaining expectation by repeatedly invoking smoothness of $V$. First,
\begin{align*}
    \E_\rho\|\nabla V \circ T^\lambda\|^2
    &\leq 2\,\E_\rho\|\nabla V\circ T^\lambda - \nabla V\circ T^\star_\diamond\|^2 + 2\, \E_\rho\|\nabla V\circ T^\star_\diamond\|^2 \\
    &\leq 2 L^2_V\, \|T^\lambda - T^\star_\diamond\|^2_{L^2(\rho)} + 2\, \E_\rho\|\nabla V\circ T^\star_\diamond\|^2 \\
    &= 2L^2_V\,W_2^2(\mu_\lambda,\pi^\star_\diamond) + 2\, \E_\rho\|\nabla V\circ T^\star_\diamond\|^2\,.
\end{align*}
For the next term, we apply the same trick, but we compare against $\pi^\star$, the true mean-field approximation:
\begin{align*}
    \E_\rho\|\nabla V\circ T^\lambda\|^2 &\leq 2L^2_V\,W_2^2(\mu_\lambda,\pi^\star_\diamond) + 4\, \E_\rho\|\nabla V\circ T^\star_\diamond - \nabla V\circ T^\star\|^2  + 4\, \E_\rho\|\nabla V\circ T^\star\|^2  \\
    &\leq 2L^2_V\, W_2^2(\mu_\lambda,\pi^\star_\diamond) + 4L_V^2\,W_2^2(\pi^\star_\diamond,\pi^\star) + 4L_V\kappa d\,,
\end{align*}
where we used \Cref{lem: mfvi_nablav_lemma} in the last step.

Our full variance bound reads
\begin{align*}
    \E[\|Q^{-1}\,(\hat{\nabla}_\lambda \cV(\mu_\lambda) - {\nabla}_\lambda \cV(\mu_\lambda))\|_Q^2] 
    &\leq 2L_V^2\Xi J\, W_2^2(\mu_{\lambda},\pi^\star_\diamond) \\
    &\qquad + 4L_V \Xi J \, (L_V \,W_2^2(\pi^\star_\diamond,\pi^\star) + \kappa d)\,. \qedhere
\end{align*}
\end{proof}

{Finally, we also prove the bound on $\Xi$ for the piecewise linear dictionary.
\begin{proof}[Proof of~\cref{lem:Xi}]
    If we can show that $Q \succeq \gamma I$ for some $\gamma > 0$, then
    \begin{align*}
        \langle Q^{-1}, \bar Q(x)\rangle
        &\le \gamma^{-1} \tr \bar Q(x)
        = \gamma^{-1}\sum_{T\in \cM} {T(x)}^2
        \le \gamma^{-1} J\,,
    \end{align*}
    where we use the fact that the elements of the piecewise linear dictionary are uniformly bounded by $1$.

    To prove the lower bound on $Q$, we note that for any $\lambda \in \R^J$,
    \begin{align*}
        \langle \lambda, Q\,\lambda\rangle
        &= \Bigl\lVert \sum_{T\in\cM}\lambda_T T\Bigr\rVert_{L^2(\rho)}^2\,.
    \end{align*}
    On an interval $[a,a+\delta]$, since $T^\lambda(x) = T^\lambda(a) + \lambda_T\,((x-a)/\delta)_+$,
    \begin{align*}
        \int_a^{a+\delta} T^\lambda(x)^2\,\rho(\mathrm d x)
        &= \int_a^{a+\delta} \Bigl(T^\lambda(a) + \lambda_T\,\frac{x-a}{\delta}\Bigr)^2\,\rho(\mathrm d x)
        \ge \lambda_T^2 \inf_{m\in \R} \int_a^b \bigl(\frac{x-a}{\delta} - m\bigr)^2\,\rho(\mathrm dx) \\
        &\gtrsim \lambda_T^2 \delta^2 \int_a^{a+\delta}\rho(\mathrm dx)\,,
    \end{align*}
    where we used the variance bound from the proof of~\cref{lem:Upsilon_lemma}.
    Summing across the intervals, we find that $\langle \lambda,Q\,\lambda \rangle \gtrsim \delta^2 \sum_{T\in\cM} \lambda_T^2$, so we can take $\gamma \asymp \delta^2$.
    This leads to an upper bound on $\Xi$ of order $\delta^{-2} \asymp J^2$.
\end{proof}
}

\section{Remaining implementation details}\label{sec:experiment_details}

\subsection{Product Gaussian mixture}
Let $V_1$ (resp.\ $V_2$) be the potential for a univariate Gaussian mixture with weights $w_{1,1}$ and $w_{1,2}$ (resp.\ $w_{2,1}$ and $w_{2,2}$) that sum to unity, and centers $m_{1,1}$ and $m_{1,2}$ (resp.\ $m_{2,1}$ and $m_{2,2}$), where all the mixture components have unit variance. Then, $V :\R^2\to\R$ defined by $V(x,y) = V_1(x) + V_2(y)$ is the potential for the Gaussian mixture with mean-weight pairs given by
\begin{align*}
    \{ &([m_{1,1}, m_{2,1}], w_{1,1}w_{2,1}), ([m_{1,1}, m_{2,2}], w_{1,1}w_{2,2}), \\
    &([m_{1,2}, m_{2,1}], w_{1,2}w_{2,1}), ([m_{1,2}, m_{2,2}], w_{1,2}w_{2,2})\}\,.
\end{align*}
We take $m_{1,1} = m_{2,1} = 2$, $m_{1,2}=m_{2,2} = -2$, with $w_{1,1} = w_{2,2} = 0.25$ and  $w_{1,2}=w_{2,1}=0.75$. As for the hyperparameters of our model, we chose $J = 28$, $\alpha=0.1$, a step-size $h = 10^{-3}$ (for both $\lambda$ and $v$), ran for 3000 iterations, and initialized at $\lambda^{(0)} = 0_{2\times J} \in \R^{2 \times J}$, and $v^{(0)} = 0_2 \in \R^2$. The KDE plots were generated via $\texttt{sklearn}$, after we generated 50,000 samples from the ground truth density and from our algorithm. 

\subsection{Non-isotropic Gaussian}\label{sec:nonisogaussian_exp}
We generated $A \in \R^{d\times d}$ with entries $A_{i,j} \sim \cN(0,1)$, and defined $\Sigma = AA^\top$ for $d=5$, which is fixed once and for all. We computed the optimal $\alpha^* = 1/\sqrt{L_V}$, since the potential is a Gaussian.  For the remaining hyper-parameters of our model, we chose $J = 28$, a step-size $h = 10^{-4}$ (for both $\lambda$ and $v$), ran for $2000$ iterations, and initialized at $\lambda^{(0)} = \bm{1}_{d \times J} \in \R^{d \times J}$, the all-ones matrix, and $v^{(0)} = 0_d \in \R^d$. At each step, we computed $\hat{\Sigma}_{\text{MF}}$ by pushing forward 10,000 samples, computing the empirical covariance, and computing the Bures--Wasserstein distance to $\Sigma_{\text{MF}}$. 

We now compute the fact that $\Sigma_{\text{MF}}$ is diagonal with components $1/(\Sigma^{-1})_{i,i}$ for $i \in [d]$. Recall the KL divergence between two Gaussians with mean zero is given by
\begin{align*}
    \kl{\cN(0,A)}{\cN(0,\Sigma)} = \frac{1}{2}\bigl[\tr(\Sigma^{-1}A)-d + \log\det(\Sigma^{-1}) - \log\det(A)\bigr]\,.
\end{align*}
Now, we impose that $A$ is a diagonal matrix with entries $A_{i,i} = a_i$ for some $a_i \geq 0$. In this case, up to constants denoted by $C$, the above reads
\begin{align*}
    \kl{\cN(0,A)}{\cN(0,\Sigma)} = \frac{1}{2}\sum_{i=1}^d \bigl[(\Sigma^{-1})_{i,i}a_i - \log(a_i)\bigr] + C\,. 
\end{align*}
Taking the derivative in $a_i$, we see that the optimality conditions yield 
\begin{align*}
    1/(\Sigma^{-1})_{i,i} = a_i^\star
\end{align*}
for every $i \in [d]$, which completes the calculation.
\subsection{Bayesian logistic regression}\label{sec:blogreg_exp}
We first randomly drew $\theta^\star \sim \cN(0,I_d)$ in $d=20$ as the ground truth parameter. Further, we let $n=100$ and randomly generated $X \in \R^{n\times d}$ as in the non-isotropic Gaussian experiment (here, $X$ takes the role of $A$), but we divided the matrix by $\lambda_{\max}(X^\top X)$ for normalization purposes. Subsequently, $Y_i$ was generated for each $i$ independently according to
\begin{align*}
    Y_i \mid X_i \sim \text{Bern}(\exp(\theta^\top X_i))\,,
\end{align*}
where $X_i$ is a row of $X$. Using this data, and assuming an improper (Lebesgue) prior on $\theta$, the potential of the posterior is given by 
\begin{align*}
    V(\theta) = \sum_{i=1}^n \bigl[\log(1 + \exp(\theta^\top X_i)) - Y_i\, \theta^\top X_i\bigr]\,.
\end{align*}

With access to $V$ and $\nabla V$, we ran standard Langevin Monte Carlo (LMC) for 5000 iterations with a step size of $h=10^{-2}$, where we generated 2000 samples.

For the hyperparameters of our model, we chose $J = 28$, $\alpha=0.1$, a step size $h = 10^{-2}$ for the $\lambda$ iterates, and $h_v = 10^{-1}$ for updating $v$, and ran for $2000$ iterations. We initialized at $\lambda^{(0)} = \bm{1}_{d \times J}/(Jd) \in \R^{d \times J}$, and $v^{(0)} = 0_d \in \R^d$. The final histograms were generated using 2000 samples from both the mean-field VI algorithm and LMC\@.

\section{Proofs for Section~\ref{sec:mixtures_of_prod}}\label{app:pf_mixtures}

In this section, we derive the gradient flows in~\Cref{sec:mixtures_of_prod}.

\begin{proof}[Proof of~\Cref{thm:wgf_mixtures}]
    We refer to~\citet[Appendix F]{lambert2022variational} for the relevant background.
    The first variation of the functional $\cF(P)\defeq \kl{\mu_P}{\pi}$ is given by
    \begin{align}\label{eq:first_var_mixtures}
        \delta \cF(P) : (\lambda, v) \mapsto \int (V + \log \mu_P + 1) \dd \mu_{\lambda, v} = \int \log \frac{\mu_P}{\pi}\dd \mu_{\lambda,v} + 1\,.
    \end{align}
    Therefore, the Wasserstein gradient is given by
    \begin{align}\label{eq:wass_grad_mixtures}
        \nabla_{\WW} \cF(P)(\lambda, v)
        &= \Bigl(Q^{-1}\,\nabla_\lambda \int \log \frac{\mu_P}{\pi}\dd\mu_{\lambda,v}, \, \nabla_v \int \log \frac{\mu_P}{\pi}\dd \mu_{\lambda,v}\Bigr)\,.
    \end{align}
    These terms are further computed as follows. First,
    \begin{align*}
        \partial_{\lambda_T} \int \log \frac{\mu_P}{\pi}\dd \mu_{\lambda,v}
        &= \partial_{\lambda_T} \int \log \frac{\mu_P}{\pi} \circ T^{\lambda,v} \dd \rho
        = \int \bigl\langle \nabla \log \frac{\mu_P}{\pi} \circ T^{\lambda,v}, T\bigr\rangle\dd \rho\,.
    \end{align*}
    Similarly, we have
    \begin{align*}
        \nabla_v \int \log \frac{\mu_P}{\pi} \dd \mu_{\lambda,v}
        &= \int \nabla \log \frac{\mu_P}{\pi} \circ T^{\lambda,v} \dd \rho\,.
    \end{align*}
    This concludes the proof.
\end{proof}

\begin{proof}[Proof of~\Cref{thm:wfr_flow}]
    This theorem follows from the expression of the first variation computed in~\eqref{eq:first_var_mixtures}, see~\citet[Appendix H]{lambert2022variational}.
\end{proof}

\end{document}